\documentclass[reqno,11pt]{amsart}
\usepackage{amsmath,amssymb,latexsym,soul,cite,mathrsfs}
\usepackage{color,enumitem,graphicx}
\usepackage[colorlinks=true,urlcolor=blue,
citecolor=red,linkcolor=blue,linktocpage,pdfpagelabels,
bookmarksnumbered,bookmarksopen]{hyperref}
\usepackage[english]{babel}
\usepackage[left=2.6cm,right=2.6cm,top=2.9cm,bottom=2.9cm]{geometry}
\usepackage[hyperpageref]{backref}

\usepackage[most]{tcolorbox}

\newcounter{mytable}

\newcounter{mytableeq}
\renewcommand{\themytableeq}{\Roman{mytableeq}}

\newcommand{\equationbox}[2]{%
  \refstepcounter{mytableeq}%
  \noindent%
  \begin{minipage}{\dimexpr\linewidth-2em}%
    #1%
  \end{minipage}%
  \hfill (\themytableeq)\label{#2}
}

\tcbset{
  boxA/.style={
    colback=gray!10,       
    colframe=black,        
    boxrule=0.5pt,         
    arc=2mm,               
    width=11.3cm,
    left=6pt, right=6pt,   
    top=6pt, bottom=6pt,
  }
}

\usepackage[utf8]{inputenc}
\usepackage{dsfont}
\usepackage{mathtools}

\usepackage{tikz}
\usetikzlibrary{arrows}

\newtheorem{theorem}{Theorem}
\newtheorem{lemma}{Lemma}

\newtheorem{proposition}{Proposition}
\newtheorem{remark}{Remark}
\newtheorem{definition}{Definition}
\newtheorem{claim}{Claim}

\newtheoremstyle{tttheorem}
{}                
{}                
{\slshape}        
{}                
{\bfseries}       
{'}               
{ }               
{}                
\theoremstyle{tttheorem}

\def\XXint#1#2#3{{\setbox0=\hbox{$#1{#2#3}{\int}$ }
		\vcenter{\hbox{$#2#3$ }}\kern-.6\wd0}}

\title[Overdetermined elliptic problems]{Overdetermined elliptic problems on model Riemannian manifolds}

\author[J.M. do \'O]{Jo\~ao Marcos do \'O*}
\author[J. de Lima]{Jaqueline de Lima}
\author[M. Santos]{M\'{a}rcio Santos}

\address[J.M. do \'O]{Department of Mathematics,
	Federal University of Para\'{\i}ba
	\newline\indent 
	58051-900, Jo\~ao Pessoa-PB, Brazil}
\email{\href{mailto:jmbo@mat.ufpb.br}{jmbo@mat.ufpb.br}}

\address[J. de Lima]{Department of Mathematics,
	Federal University of Para\'{\i}ba
	\newline\indent 
	58051-900, Jo\~ao Pessoa-PB, Brazil}
\email{\href{mailto:jaqueline.lima@academico.ufpb.br}{jaqueline.lima@academico.ufpb.br}}

\address[M. Santos]{Department of Mathematics,
	Federal University of Para\'{\i}ba
	\newline\indent 
	58051-900, Jo\~ao Pessoa-PB, Brazil}
\email{\href{mailto:marcio.santos@academico.ufpb.br}{marcio.santos@academico.ufpb.br}}

\thanks{*Corresponding author.}


\begin{document}

\begin{abstract}
We establish a rigidity theorem for annular sector-like domains in the setting of overdetermined elliptic problems on model Riemannian manifolds. Specifically, if such a domain admits a solution to the inhomogeneous Helmholtz equation satisfying both constant Dirichlet and constant Neumann boundary conditions, then the domain must be a spherical sector, and the solution must be radially symmetric. This result underscores the strong geometric constraints imposed by overdetermined boundary conditions, extending classical rigidity phenomena to this more general framework.
\end{abstract}
	
	\maketitle
	
	

\section{Introduction}
This paper examines a geometric problem first introduced by Serrin~\cite{Serrin1971}, originally motivated by a fluid dynamics question posed by Fosdick. In~\cite{Serrin1971}, Serrin studied the rigidity of solutions to the following overdetermined Poisson problem:
\begin{equation} \label{serrin_system}
    \begin{cases}
        \Delta u = -1 & \text{in } \Omega \\[8pt]
        u = 0          & \text{on } \partial\Omega \\[8pt]
        \dfrac{\partial u}{\partial \nu} = -c & \text{on } \partial\Omega
    \end{cases}
\end{equation}
where $\Omega \subset \mathbb{R}^n$ is a smooth bounded domain (i.e., a connected open set), $\nu$ denotes the outward unit normal vector, $\partial u / \partial \nu$ is the normal derivative of $u$ on $\partial\Omega$, and $c > 0$ is a constant.
This seminal study established that Problem \( \eqref{serrin_system} \) has a solution if and only if the domain \( \Omega \) is a ball and the function \( u \) is radial. The proof was primarily based on the \emph{method of moving planes}, a technique originally introduced by A. Alexandrov \cite{alexandrov1958vestnik,alexandrov1962characteristic} and now widely recognized by this name.

Alternatively, Weinberger~\cite{MR333221} introduced a different approach based on a $P$-function associated with the solution $u$ of \eqref{serrin_system}. His method combines classical maximum principles with the Pohozaev identity for Euclidean domains to establish the result. While Weinberger's approach is elegant and simple, it relies crucially on the linearity of the Laplace operator, which limits its extension to nonlinear problems. Moreover, it applies only to Poisson equations with constant right-hand sides. In contrast, Serrin's method of moving planes provides a more versatile framework that can handle general fully nonlinear elliptic operators with arbitrary data, demonstrating significantly broader applicability.

Following Serrin's pioneering work, numerous extensions of his symmetry result have been established for various space forms, as demonstrated in \cite{CV, ciraolo2019serrin,fall2018serrin, kumaresan1998serrins, molzon1991symmetry,qiu2017overdetermined,roncoroni2018serrin} and related literature.

Recently, significant research has been devoted to the rigidity of overdetermined equations associated with various classes of operators and domains, particularly emphasising their mathematical properties and intrinsic structural constraints. Given its deep connection to geometric analysis, this topic remains a relevant and compelling study area. 

We direct the reader to \cite{PacellaTralli2020} and its references for a thorough summary.  
The regularity theory and maximum principles for second-order elliptic equations are classically treated in Gilbarg and Trudinger \cite{MR1814364}. On Riemannian manifolds, the analogous framework—incorporating Sobolev spaces and geometric constraints—is detailed in Hebey’s monograph \cite{MR1688256}.

In their work \cite{PacellaTralli2020}, F. Pacella and G. Tralli analyzed a partially overdetermined problem 
defined in a sector-like domain $\Omega$ in an open convex cone $\Sigma$  in Euclidean space $\mathbb{R}^n$: 
\begin{equation}\label{eq:pde-system}
\begin{cases}
\Delta u = -1 & \text{in } \Omega \\
u = 0 & \text{on } \partial\Omega \\
\frac{\partial u}{\partial\nu} = -c & \text{(constant) on } \Gamma_0 \text{ where } c > 0 \\
\frac{\partial u}{\partial\nu} = 0 & \text{on } \Gamma_1 \setminus \{O\}
\end{cases}
\end{equation}
where $\nu = \nu(x)$ denotes the outward unit normal to $\partial\Omega$ (defined for all $x \in \Gamma_0 \cup \Gamma_1 \setminus \{O\}$). Under a convexity assumption on the cone, among other results, they derived a Serrin-type rigidity result, demonstrating that solutions exist exclusively for spherical sectors.  Recently, G. Ciraolo and A. Roncoroni  \cite{ciraolo2020serrin} generalized this result to space forms. Specifically, they studied the following partially overdetermined problem in space forms:
\begin{equation}\label{eq:overdetermined}
\begin{cases}
\Delta u + nKu = -n & \text{in } \Omega \\
u = 0 & \text{on } \partial\Omega \\
\frac{\partial u}{\partial\nu} = -c & \text{(constant) on } \Gamma_0 \\
\frac{\partial u}{\partial\nu} = 0 & \text{on } \Gamma_1 \setminus \{O\}
\end{cases}
\end{equation}
where $K \in \{0, +1, -1\}$ corresponds to Euclidean space, the upper unit hemisphere $\mathbb{S}^n_+$, and hyperbolic space $\mathbb{H}^n$, respectively.  Their main result states that if \( \Omega \) is a sector-like domain within a convex cone in these space forms, then \( \Omega \) must be the intersection of the cone with a geodesic ball, and the solution \( u \) must be radially symmetric.

Subsequently, J.~Lee and K.~Seo in  \cite{lee2023radial} generalized this result by relaxing the requirement that $\Omega \subset \mathbb{S}^{n}_{+}$. Instead, they considered $\Omega \subset \mathbb{S}^n$ to be a \textit{star-shaped domain with respect to a pole} $\mathcal{O}$, meaning that each portion of $\partial\Omega$ can be represented as a graph over a geodesic sphere centered at $\mathcal{O}$. This weaker geometric assumption permits the domain to extend beyond the hemisphere while preserving a controlled symmetry about the point $\mathcal{O}$.

Subsequently, J.~Lee and K.~Seo in~ \cite{lee2023radial} extended this result by relaxing the condition that \( \Omega \subset \mathbb{S}^n_+ \). Instead, they considered domains \( \Omega \subset \mathbb{S}^n \) that are \textit{star-shaped with respect to a pole} \( \mathcal{O} \); that is, each portion of \( \partial\Omega \) can be described as the graph of a function over a geodesic sphere centered at \( \mathcal{O} \). This weaker geometric assumption allows the domain to extend beyond the hemisphere while maintaining a controlled symmetry about the point \( \mathcal{O} \).

More recently, in \cite{lee2023overdetermined}, the authors investigated the equation
\[
\Delta u = -n - nK u
\]
in annular domains contained in space forms, that is, domains bounded by two surfaces, one of which is a geodesic hypersphere. In this setting, it is assumed that the domain includes a spherical boundary component where Dirichlet conditions are imposed, while Neumann conditions are prescribed on the remaining portion of the boundary. Using variational techniques and geometric identities adapted to the constant curvature setting, J.~Lee and K.~Seo show that, under these assumptions, the domain must necessarily be a geodesic annular region bounded by two concentric hyperspheres. This result reinforces the geometric rigidity imposed by the overdetermined structure of the equation and extends the symmetry characterization of admissible domains, even when the domain is not simply connected or deviates from a fully radial structure.

%
%

\subsection{Our problem}
Let $(\mathcal M, g)$ be a complete Riemannian manifold with dimension $N$. 
The class of Riemannian model $(\mathcal{M}, g)$ includes the classical space forms. Precisely, a manifold $\mathcal{M}$ of dimension $N\geq 2$ admitting a pole $\mathcal{O}$ and whose metric $g$ is given, in polar coordinates around $\mathcal{O}$, by
\begin{equation}\label{08}
\mathrm{d} s^2=\mathrm{d} r^2+\psi(r)^2\mathrm{d} \theta^2 \quad \text{for }r \in (0,R)\text{ and }\theta\in\mathbb{S}^{N-1},
\end{equation}
where $r$ is by construction the Riemannian distance between the point $P=(r,\theta)$ to the pole $\mathcal{O}$, $\psi$ is a smooth positive function in $(0,R)$ and $\mathrm{d}\theta^2$ is the canonical metric on the unit sphere $\mathbb{S}^{N-1}$. Note that our results apply to the important case of space forms, i.e., the unique complete and simply connected Riemannian manifold of constant sectional curvature $K_\psi$ corresponding to the choice of $\psi$, namely,

\medskip

\equationbox{%
\begin{tcolorbox}[boxA]
    \begin{tabular}{l l l l }
        {\bf Euclidean space:}   &  $ \psi(r) = r,  $ &  $K_\psi = 0,$ & $r\in [0,\infty)$ \\
        {\bf Elliptic space: }   &  $ \psi(r) = \sin r,  $ &  $K_\psi = 1,$ & $ r\in [0,\pi)$ \\
        {\bf Hyperbolic space:}  &  $ \psi(r) = \sinh r, $ &  $ K_\psi = -1,$ &  $r\in [0,\infty)$\\
    \end{tabular}
\end{tcolorbox}
}{eq:tabela}

\medskip

\begin{definition}
    We define an \emph{open cone} $\Sigma$ with vertex at $\left\{\mathcal{O}\right\}$ as the set
$$\Sigma \doteq \left\{tx;\\ x\in\omega,\\ t\in I\right\},$$
for some open domain $\omega\subset \mathbb{S}^n$, where $I=[0,\infty)$ for the Euclidean space and Hyperbolic space, and $I=[0,\pi)$ for  Elliptic space.
Here, if \( \nu \) denotes the outward unit normal vector field, the second fundamental form is given by $A(X, Y) = g(\nabla_{X} \nu, Y),$ where \( X \) and \( Y \) are tangent vector fields to \( \partial \Sigma \). We say that a cone is \emph{convex} if its second fundamental form is nonnegative at every point \( x \in \partial \Sigma \), that is, $A(X,X)\geq 0$ for all tangent vector fields to $\partial\Sigma$. Moreover, the mean curvature $H$ is the trace of the second fundamental form.
\end{definition}

In the same way as \cite{ciraolo2020serrin} and \cite{PacellaTralli2020} we have the concept of \emph{annular sector-like domain} that we define below.
\begin{definition}\label{def_sec}
    Let $\Sigma$ be a convex open cone as described above. Consider bounded domains $\Omega_1$ and $\Omega_2$ contained in $\Sigma$, with $\overline{\Omega}_2 \cap \Sigma \subset \Omega_1$ and a point $\mathcal{O} \in \partial \Omega_2$. Denote by $\mathcal{H}_{n-1}(\cdot)$ the $(n-1)$-dimensional Hausdorff measure.
We call $\Omega := \Omega_1 \setminus \overline{\Omega}_2$ an \textbf{annular sector-like domain} if the following conditions are satisfied: $\Gamma_0 := \partial \Omega_1 \cap \Sigma$, $\Gamma_2 := \partial \Omega_2 \cap \Sigma$, and $\Gamma_1 := \partial \Omega \setminus (\overline{\Gamma}_0 \cup \overline{\Gamma}_2)$ satisfy $\mathcal{H}_{n-1}(\Gamma_i) > 0$ for $i = 0, 1, 2$; each of $\Gamma_0$, $\Gamma_1$, and $\Gamma_2$ is a smooth $(n-1)$-dimensional manifold; the boundaries $\partial \Gamma_0$ and $\partial \Gamma_2$, which are contained in $\partial \Omega$, are smooth $(n-2)$-dimensional manifolds; and the boundary of $\Gamma_1$ satisfies $\partial \Gamma_1 = \partial \Gamma_0 \cup \partial \Gamma_2$.
\end{definition}

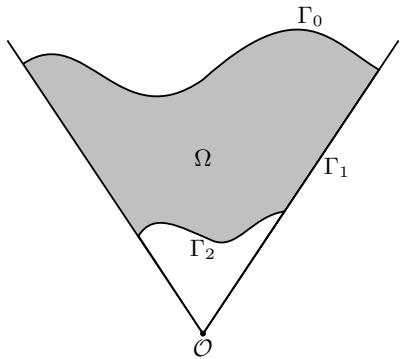
\begin{center}
        \begin{tikzpicture}[scale=1.3]

\draw[line width=0.7pt,color=black] (2,3)--(0,0)-- (-2,3);

\draw[line width=0.7pt,color=black,fill=black,fill opacity=0.25] 
            (-1.84,2.76) .. controls (-1.2,3.2) and (-0.85,2) ..
            (0,2.6).. controls (1,3.5) and (1.3,3) .. 
            (1.8,2.7)--(0,0);

\draw[line width=0.7pt,color=black,fill=white,fill opacity=1] 
            (-0.66,1) .. controls (-0.5,1.25) and (-0.25,1.1) ..
            (0.1,0.95).. controls (0.35,0.85) and (0.5,1.2) .. 
            (0.83,1.25)--(0,0)--(-0.66,1);

\draw [fill=black] (0,0) circle (0.7pt);

\draw  (0,-0.15) node {\footnotesize \textcolor{black}{$\mathcal{O}$}};
      
\draw  (1.37,1.7) node {\footnotesize \textcolor{black}{$\Gamma_1$}};

\draw  (1.1,3.25) node {\footnotesize \textcolor{black}{$\Gamma_0$}};

\draw  (0,1.8) node {\footnotesize \textcolor{black}{$\Omega$}};

\draw  (0,0.86) node {\footnotesize\textcolor{black}{$\Gamma_2$}};

\end{tikzpicture}
\end{center}

Inspired by \cite{lee2023overdetermined}, let us consider the following Serrin problem in annular sector-like domains involving  Helmholtz type equations as follows
\begin{equation}\label{PG}
\left\{
    \begin{aligned}
&        \Delta u = - k n u -n,\quad & \text{in} \quad \Omega \\
&        u=0, \quad \dfrac{\partial u}{\partial \nu} = c_0 \quad & \text{on} \quad \Gamma_0\\
&        \dfrac{\partial u}{\partial \nu}=0 \quad & \text{on} \quad \Gamma_1\\
&        u=a, \quad  \dfrac{\partial u}{\partial \nu}=c_2  \quad & \text{on} \quad \Gamma_2
\end{aligned}
    \right.
\end{equation}
where $c_0,c_2,a\in\mathbb{R}$ and $k$ is the sectional curvature of  $M$.

\subsection{Main results} 

\subsubsection{Euclidean space}

\begin{theorem}[Euclidean space]\label{Euclidean space}
Let $\Omega$ be an annular sector-like domain contained in $\mathbb{R}^n$, with $\Omega_2=B_R(\mathcal{O})$. Given $0<R<R_1$, suppose there is a solution $u \in C^2(\Omega) \cap C^1(\overline{\Omega}) \cap W^{1,\infty}(\Omega) \cap W^{2,2}(\Omega) $ such that
\begin{equation}\label{saber}
    \begin{cases}
        \Delta u=-n & \text{in } \Omega\\
        u=0 \text{ and } \dfrac{\partial u}{\partial \nu }=-R_1 & \text{on } \Gamma_0\\
        \dfrac{\partial u}{\partial \nu} =0  & \text{on } \Gamma_1\\
        u=a>0 \text{ and } \dfrac{\partial u}{\partial \nu}= R & \text{on } \Gamma_2
    \end{cases}
\end{equation}
where \( R_1 \) and \( a \) are constants.  Then, \( \Gamma_0 = \partial B_{R_1} \cap \Sigma \), that is, \( \Omega = \{x\in \Sigma: R< d(x,\mathcal{O})< R_1\} \).  Moreover, \( u \) is a radial function, precisely,
\begin{equation*}
    u(x) = \frac{R_1^2}{2} - \frac{r^2}{2}, \quad \text{where} \quad  r=\operatorname{dist}(x,\mathcal{O}).
\end{equation*}
\end{theorem}

\subsubsection{Elliptic space}
We now present two rigidity results concerning annular sector-like domains contained in the sphere. The following definition will be used in formulating these results.

\begin{definition}
An annular sector-like domain is called star-shaped with respect to $p$ if each component of the boundary $\partial\Omega\setminus\Gamma_1$ can be written as a graph over geodesic sphere with center $p.$ 
\end{definition}

\begin{definition}
An annular sector-like domain $\Omega$ is said to be \textbf{star-shaped with respect to a point} $p$ if each component of the boundary $\partial \Omega \setminus \Gamma_1$ can be represented as the graph of a function over a geodesic sphere centered at $p$.
\end{definition}

In particular, from above definition the angle function $\langle\nu,\partial_r\rangle$ does not change sign along of the boundary.

\begin{theorem}[Case: $\Omega$ is star-shaped]\label{cajarana}
Let $\Omega$ be an annular sector-like domain contained in $\mathbb{S}^n$ , with $\Omega_2=B_R(\mathcal{O})$ and $\Omega_2$ contained in $\mathbb{S}^n_+$.
 Given $0<R<R_1<{\pi}$, suppose there is a solution $u \in C^2(\Omega) \cap C^1(\overline{\Omega}) \cap W^{1,\infty}(\Omega) \cap W^{2,2}(\Omega) $ such that
\begin{equation}\label{saber2}
        \begin{cases}
            \Delta u=-nu-n & in \ \Omega\\
            u=0 \quad  \dfrac{\partial u}{\partial \nu }=-\frac{\sin R_1}{\cos R_1} & on \ \Gamma_0\\
            \dfrac{\partial u}{\partial \nu} =0  & on \ \Gamma_1\\
            u=a<-1 \quad  \dfrac{\partial u}{\partial \nu}= \frac{\sin R}{\cos R_1}<0 & on \ \Gamma_2
        \end{cases}
        \end{equation}
      where \( R_1 \) and \( a \) are constants. 
If $\Omega$ is star-shaped with respect to $\mathcal{O}$, then  \( \Gamma_0 = \partial B_{R_1} \cap \Sigma \), that is, \( \Omega = \{x\in \Sigma: R< d(x,\mathcal{O})< R_1\} \).  Moreover, \( u \) is a radial function, precisely,
\begin{equation}\label{radial}
    u(x) = \frac{\cos r-\cos R_1}{\cos R_1}, \quad \text{where} \quad  r=\operatorname{dist}(x,\mathcal{O}).
\end{equation}      
\end{theorem}

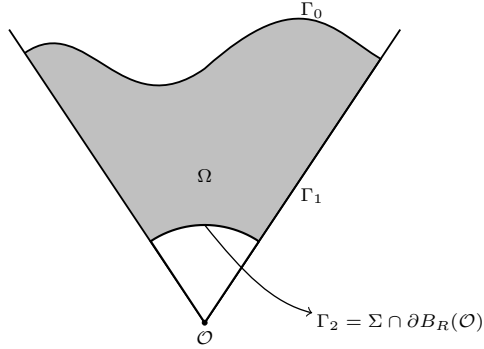
 \begin{center}
        \begin{tikzpicture}[scale=1.3]

\draw[line width=0.7pt,color=black] (2,3)--(0,0)-- (-2,3);

\draw[line width=0.7pt,color=black,fill=black,fill opacity=0.25] 
            (-1.84,2.76) .. controls (-1.2,3.2) and (-0.85,2) ..
            (0,2.6).. controls (1,3.5) and (1.3,3) .. 
            (1.8,2.7)--(0,0);

\draw [color=black,fill=white,fill opacity=1,thick,domain=56:123.5] plot ({cos(\x)}, {sin(\x)})--(0,0)--(0.553,0.83);
            
\draw [fill=black] (0,0) circle (0.7pt);
            
\draw[line width=0.5pt,color=black,->] (0,1) .. controls (0.4,0.5) and (0.7,0.2) ..  (1.1,0.1);

\draw  (0,-0.15) node {\tiny \textcolor{black}{$\mathcal{O}$}};
      
\draw  (1.1,1.3) node {\tiny \textcolor{black}{$\Gamma_1$}};

\draw  (1.1,3.2) node {\tiny \textcolor{black}{$\Gamma_0$}};

\draw  (0,1.5) node {\tiny \textcolor{black}{$\Omega$}};

\draw  (2,0) node {\tiny \textcolor{black}{$\Gamma_2=\Sigma\cap\partial B_R(\mathcal{O})$}};

\end{tikzpicture}
\end{center}

Our next theorem considers an annular sector-like domain in $\mathbb{S}^n_+$ without any conditions on $\Gamma_0$, unlike the previous result.

\begin{theorem}\label{general case}
Let $\Omega$ be an annular sector-like domain contained in $\mathbb{S}^n_+$, with $\Omega_2=B_R(\mathcal{O})$. Given $0<R<R_1<\dfrac{\pi}{2}$, suppose there is a solution $u \in C^2(\Omega) \cap C^1(\overline{\Omega}) \cap W^{1,\infty}(\Omega) \cap W^{2,2}(\Omega) $ such that
    \begin{equation*}
\begin{cases}
    \Delta u = -n u - n, & \quad \text{in } \Omega \\[8pt]
    u = 0, \quad \dfrac{\partial u}{\partial \nu} = -\dfrac{\sin R_1}{\cos R_1}, & \quad \text{on } \Gamma_0 \\[10pt]
    \dfrac{\partial u}{\partial \nu} = 0, & \quad \text{on } \Gamma_1\\[10pt]
    u = a > 0, \quad \dfrac{\partial u}{\partial \nu} = \dfrac{\sin R}{\cos R_1}, & \quad \text{on } \Gamma_2.
\end{cases}
\end{equation*}
Then, \( \Gamma_0 = \partial B_R \cap \Sigma \), that is, \( \Omega = \{x\in \Sigma: R< d(x,\mathcal{O})< R_1\} \).  Moreover, \( u \) is a radial function, precisely,
\begin{equation*}
    u(x) = \frac{\cos r-\cos R_1}{\cos R_1}, \quad \text{where} \quad  r=\operatorname{dist}(x,\mathcal{O}).
\end{equation*}    
\end{theorem}

\subsubsection{Hyperbolic space} Let $\Omega$ be an {\it annular sector-like domain} in the Hyperbolic space as in Definition~\ref{def_sec}.

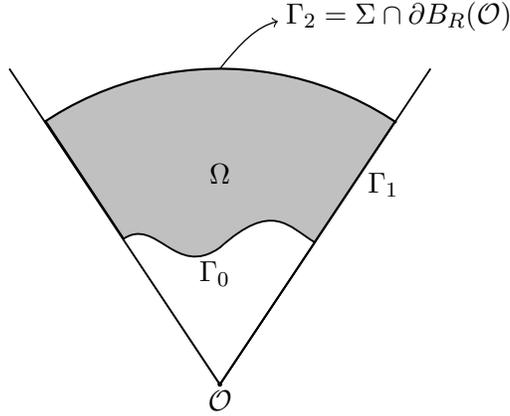
\begin{center}
        \begin{tikzpicture}[scale=0.7]

\draw [line width=0.7pt,color=black,fill=black,fill opacity=0.25,thick,domain=56:123.5] plot ({6*cos(\x)}, {6*sin(\x)})--(0,0)--(6*0.553,6*0.83);
            
\draw[line width=0.7pt,color=black,fill=white,fill opacity=1] 
            (-1.84,2.76) .. controls (-1.2,3.2) and (-0.85,2) ..
            (0,2.6).. controls (1,3.5) and (1.3,3) .. 
            (1.8,2.7)--(0,0);
        
\draw[line width=0.7pt,color=black] (4,6)--(0,0)-- (-4,6);

\draw [fill=black] (0,0) circle (1.0pt);

\draw[line width=0.5pt,color=black,->] (0,1+5) .. controls (0.4,1.5+5) and (0.7,1.8+5) ..  (1.1,1.9+5);

\draw  (0,-0.3) node { \textcolor{black}{$\mathcal{O}$}};
            
\draw  (3.1,3.8) node { \textcolor{black}{$\Gamma_1$}};

\draw  (-0.1,2.1) node {\textcolor{black}{$\Gamma_0$}};
                        
\draw  (0,4) node { \textcolor{black}{$\Omega$}};

\draw  (3.4,7) node { \textcolor{black}{$\Gamma_2=\Sigma\cap\partial B_R(\mathcal{O})$}};

\end{tikzpicture}
    \end{center}

\begin{theorem}\label{jenipapo}  Let $\Omega$ be an annular sector-like domain contained in $\mathbb{H}^n$, with $\Omega_2=B_R(\mathcal{O})$. Given $0<R_1<R$, suppose there is a solution $u \in C^2(\Omega) \cap C^1(\overline{\Omega}) \cap W^{1,\infty}(\Omega) \cap W^{2,2}(\Omega) $ such that

\begin{equation}\label{saber3}
        \begin{cases}
            \Delta u=nu-n & in \ \Omega\\
            u=0 \ and \  \dfrac{\partial u}{\partial \nu }=\frac{\sinh R_1}{\cosh R_1} & on \ \Gamma_0\\
            \dfrac{\partial u}{\partial \nu} =0  & on \ \Gamma_1\\
            u=a \ and \  \dfrac{\partial u}{\partial \nu}=- \frac{\sinh R}{\cosh R_1} & on \ \Gamma_2
        \end{cases}
        \end{equation}
      where \( R_1 \) and \( a \) are constants. If $1-\cosh{R}/\cosh{R_1}\leq a<0$,
 then \( \Gamma_0 = \partial B_R \cap \Sigma \), that is, \( \Omega = \{x\in \Sigma: R_1< d(x,\mathcal{O})< R\} \).  Moreover, \( u \) is a radial function, precisely,
\begin{equation*}
    u(x) = \frac{\cosh R_1 -\cosh r}{\cosh R_1}, \quad \text{where} \quad  r=\operatorname{dist}(x,\mathcal{O}).
\end{equation*}      
\end{theorem}

In the previous theorem, we impose the following hypothesis on the constant $a$:
\[ 1 - \frac{\cosh R}{\cosh R_1} \leq a < 0, \]
which gives the value of the Dirichlet boundary condition for the solution $u$ on $\Gamma_2$. We emphasize that in our subsequent theorem, for star-shaped annular domains, this constraint on $a$ can be relaxed.

\begin{theorem}\label{jenipapo2}  Let $\Omega$ be an annular sector-like domain contained in $\mathbb{H}^n$, with $\Omega_2=B_R(\mathcal{O})$. Assume that $\Omega$ is star-shaped with respect to $\mathcal{O}$.  Given $0<R_1<R$, suppose there is a solution $u \in C^2(\Omega) \cap C^1(\overline{\Omega}) \cap W^{1,\infty}(\Omega) \cap W^{2,2}(\Omega) $ such that

\begin{equation}\label{saber4}
        \begin{cases}
            \Delta u=nu-n & in \ \Omega\\
            u=0 \ and \  \dfrac{\partial u}{\partial \nu }=\frac{\sinh R_1}{\cosh R_1} & on \ \Gamma_0\\
            \dfrac{\partial u}{\partial \nu} =0  & on \ \Gamma_1\\
            u=a<0 \ and \  \dfrac{\partial u}{\partial \nu}=- \frac{\sinh R}{\cosh R_1} & on \ \Gamma_2
        \end{cases}
        \end{equation}
      where \( R_1 \) and \( a \) are constants. Then \( \Gamma_0 = \partial B_{R_1} \cap \Sigma \), that is, \( \Omega = \{x\in \Sigma: R_1< d(x,\mathcal{O})< R\} \).  Moreover, \( u \) is a radial function, precisely,
\begin{equation*}
    u(x) = \frac{\cosh R_1 -\cosh r}{\cosh R_1}, \quad \text{where} \quad  r=\operatorname{dist}(x,\mathcal{O}).
\end{equation*}      
\end{theorem}

%
%

\subsection{The strategy of the Proof}
To prove the main result in the Euclidean setting, we first establish a maximum principle (Proposition~\ref{MP}), which shows that the solution \(u\) of \eqref{saber}  is strictly positive in \(\Omega\). Next, we introduce the auxiliary function
\[
f(x) \;=\; u(x) \;+\; \tfrac12\,|x - \mathcal{O}|,
\]
where \(\mathcal{O}\in\mathbb{R}^n\) is the vertex of the cone.  Applying the polarized Bochner formula to \(f\) and \(u\), one obtains
\begin{equation}\label{oxe}
\Delta \bigl\langle \nabla f,\nabla u \bigr\rangle \;\ge\; 0,
\end{equation}
with equality if and only if the Hessian of \(u\) is proportional to the Euclidean metric \(g\), namely
\[
\nabla^2 u \;=\; \frac{\Delta u}{n}\,g \;=\;-\,g.
\]
By using a version of the divergence theorem adapted to an annular sector-like domain (Theorem~\ref{DT}), we show
\[
\int_{\Omega} u\,\Delta\bigl\langle \nabla f,\nabla u \bigr\rangle \;\le\; 0.
\]
Since \(u>0\) in \(\Omega\) and \eqref{oxe} asserts nonnegativity of the integrand, it follows that 
\[
\Delta \bigl\langle \nabla f,\nabla u \bigr\rangle \;=\; 0 
\quad\text{in }\Omega,
\]
and hence \(\nabla^2 u=-g\).  Finally, invoking the Obata–type rigidity result (Lemma~\ref{Obata}) yields the desired characterization of \(u\).

We now present the argument used to prove the main results in the elliptic and hyperbolic cases for star-shaped domains, namely Theorems~\ref{cajarana} and~\ref{jenipapo2}.  We begin by introducing the auxiliary \( P \)-function
\[
    P \coloneqq |\nabla u|^2 + 2u + k u^2,
\]
where \( k \) denotes the constant sectional curvature of the ambient space. The key objective is to show that \( P \) is constant. Indeed, if \( P \) is constant, then \( \Delta P = 0 \), which in turn leads to the identity
\[
    \nabla^2 u = (-1 - k u)g.
\]
Finally, invoking Lemma~\ref{Obata}, we obtain the desired conclusion. In Propositions~\ref{caja} and \ref{egua}, we show that \(P\) satisfies the hypotheses of the maximum principle (Proposition~\ref{MP}).  It follows that \(P\) attains its maximum on $\Gamma_0 \cup \Gamma_2$. Indeed, using the Hopf Lemma we conclude that \(P\) attains its maximum on $\Gamma_0$. Assume, by contradiction, that \( P \) is not constant. We define an auxiliary function \( \widetilde{P} \) by
\[
    \widetilde{P} \coloneqq \langle \nabla u, \nabla \psi' \rangle + k u \psi' + \psi',
\]
where \( \psi \) was defined in Table~\eqref{eq:tabela}. Suppose that \( \widetilde{P} \) is also not constant. Under this assumption, we arrive at a contradiction. Therefore, at least one of \( P \) or \( \widetilde{P} \) must be constant.  However, we observe that the normal derivatives of \( P \) and \( \widetilde{P} \) vanish on the regions \( \Gamma_0 \) and \( \Gamma_2 \), sharing common vanishing factors. This implies that \( P \) itself must be constant. Indeed, by the Hopf Lemma, the normal derivative of \( P \) on \( \Gamma_2 \) would otherwise be strictly positive, contradicting the shared vanishing property. Thus, \( P \) is necessarily constant.

To prove Theorem~\ref{general case}, we proceed as before to prove that the maximum of \( P \) is attained on \( \Gamma_0 \). Thus, we have \( P < c^2 \) in the domain, where \( c \) denotes the constant value of \( P \) on \( \Gamma_0 \). Multiplying both sides of this inequality by the conformal factor associated with the conformal and closed vector field \( X = \psi \, \partial_r \), and integrating over the domain, we obtain
    \[
    \int_{\Omega} \psi' P \, dV < c^2 \int_{\Omega} \psi' \, dV.
    \]
Applying a Pohozaev-type identity, Proposition~\ref{a11}, to the left-hand side, we derive a contradiction. Therefore, \( P \) must be constant. The conclusion then follows as in the previous case.

To prove Theorem~\ref{jenipapo}, we proceed as before by applying Proposition~\ref{MP} to show that the maximum of \( P \) is attained on \( \Gamma_0 \cup \Gamma_2 \). Since
\[
P|_{\Gamma_0} = \frac{\sinh^2 R_1}{\cosh^2 R_1} \leq P|_{\Gamma_2} = \frac{\sinh^2 R}{\cosh^2 R_1} + 2a + a^2,
\]
and using the fact that \( 1 - \frac{\cosh R}{\cosh R_1} \leq a < 0 \), we conclude that the maximum of \( P \) is attained on \( \Gamma_2 \).
Suppose, by contradiction, that \( P \) is not constant. Then, by the Hopf lemma, it follows that
$
\frac{\partial P}{\partial \nu} > 0 \quad \text{on } \Gamma_2,
$
which contradicts our hypothesis that
$
\frac{\partial P}{\partial \nu} \leq 0 \quad \text{on } \Gamma_2.
$

%
%

\subsection{Organization of the Paper} 
Section~2 is devoted to the presentation of a version of the maximum principle adapted to our geometric setting, which constitutes a fundamental tool employed throughout the paper. In Section~3, we establish rigidity results in space forms. Section~3.1 contains the proof of Theorem~1. Section~3.2 addresses the spherical setting: Section~3.2.1 is concerned with the proof of Theorem~2; Section~3.2.2 provides a preparatory discussion that anticipates the general case and lays the groundwork for the subsequent arguments; and Section~3.2.3 concludes with the proof of the general result stated in Theorem~3. Section~3.3 is dedicated to the hyperbolic setting, with Section~3.3.1 presenting the proof of Theorem~4, and Section~3.3.2 concluding with the proof of Theorem~5.

\subsection{Notations}
\begin{itemize}
  \item $\nu$ denotes the outward unit normal vector on the boundary $\partial \Omega$.
  
  \item $u_\nu = \dfrac{\partial u}{\partial \nu}$ denotes the normal derivative of a function $u$ on $\partial \Omega$.
  
  \item $\chi_E$ denotes the characteristic function of a measurable set $E \subset \mathbb{R}^n$, that is,
  \[
    \chi_E(x) = 
    \begin{cases}
      1, & \text{if } x \in E, \\
      0, & \text{if } x \notin E.
    \end{cases}
  \]
  
  \item For a measurable set $\Omega \subset \mathbb{R}^n$, $|\Omega|$ denotes its Lebesgue measure.
  
  \item $C^{k}(\Omega)$ denotes the space of real-valued functions on $\Omega$ that are continuously differentiable up to order $k$.
  
  \item For $1 \leq p < \infty$, the Lebesgue space $L^p(\Omega)$ is defined as
  \[
    L^p(\Omega) = \left\{ u : \Omega \to \mathbb{R} \,\middle|\, u \text{ is measurable and } \int_\Omega |u(x)|^p\, \mathrm{d} x < \infty \right\}.
  \]  
  \item $W^{1,p}(\Omega)$ denotes the Sobolev space on $\Omega$, that is, the space of functions $u \in L^p(\Omega)$ such that all weak first-order partial derivatives $\partial_i u$ also belong to $L^p(\Omega)$, for $i = 1, \dots, n$.
  \item Denote by $\mathcal{H}_{n-1}(\cdot)$ the $(n-1)$-dimensional Hausdorff measure.
\end{itemize}

    \section{Maximum principle}\label{preliminaries} 
Firstly, inspired by ideas due to Pacella and Tralli, \cite{PacellaTralli2020}, let us provide the following maximum principle addressed to annular sector-like domains contained in a space form.

 \begin{proposition}\label{MP} If $v \in C^2(\Omega) \cap C^1 (\Gamma_0 \cup \Gamma_1 \cup \Gamma_2) \cap W^{1,2}(\Omega)$ is a  classical solution of 
    \begin{equation}
        \begin{cases}
            \Delta v \geq 0 & in \ \Omega\\
            v=c_0 &on \ \Gamma_0\\
           \dfrac{\partial v}{\partial \nu}  \leq 0 & on \ \Gamma_1\\
             v=c_2 & on \ \Gamma_2
        \end{cases}
    \end{equation}
    Then 
    \begin{equation*}
        \max_{x\in \overline{\Omega}} v(x) = c_3:=\max\{c_0,c_2\}
    \end{equation*}
\end{proposition}
\begin{proof}
    Consider the auxiliary function $\xi=v-c_3$. We will prove that $\xi^+\equiv 0$ and consequently $v \leq c_3$.
We have a lack of regularity for the vector field $\nabla v$ at the non-regular part of $\partial \Omega$, that is, at the  $\overline{\Gamma}_1$ and $\overline{\Gamma}_2$ intersect, and   $\overline{\Gamma}_1$ and $\overline{\Gamma}_0$ intersect .Let us write $\Omega=\Omega_1 \cup \Omega_2$ where $\Gamma_2 \subset \partial\Omega_1$ and 
$\overline{\Gamma}_0 \subset \partial\Omega_2$. For $\varepsilon_1 , \, \varepsilon_2 >0$ sufficiently small, we define, 
\begin{equation*}
\begin{aligned}
\Omega_{1,\varepsilon_1}=\{ x\in \Omega_1 : d(x, \partial \Gamma_2)>\varepsilon_1\} \\
\Omega_{2,\varepsilon_2}=\{ x\in \Omega_2 : d(x, \partial \Gamma_0)>\varepsilon_2\}
\end{aligned}
\end{equation*}
Using the divergence theorem, we obtain 
\begin{align*}
     0 &\leq  \int_{\Omega_{1,\varepsilon_1}} \xi^+\Delta v \, \mathrm{d} V + \int_{\Omega_{2,\varepsilon_2}} \xi^+\Delta v \, \mathrm{d}V\\
    &=  \int_{\Omega_{1,\varepsilon_1}} \operatorname{div}\left( \xi^+\nabla v\right)\, \mathrm{d}V-\int_{\Omega_{1,\varepsilon_1}} \left\langle \nabla \xi^+,\nabla v \right\rangle \, \mathrm{d}V
    + \int_{\Omega_{2,\varepsilon_2}}  \operatorname{div}\left( \xi^+\nabla v \right)\, \mathrm{d}V-\int_{\Omega_{2,\varepsilon_2}}  \left\langle \nabla \xi^+,\nabla v \right\rangle \, \mathrm{d}V \\
    & = \int_{\partial \Omega_{1,\varepsilon_1}} \xi^+\dfrac{\partial v}{\partial\nu}\, \mathrm{d}\sigma-\int_{\Omega_{1,\varepsilon_1}\cap\{h^+>0\}} |\nabla \xi^+|^2\, \mathrm{d}V+
    \int_{\partial \Omega_{2,\varepsilon_2}} h^+\dfrac{\partial v}{\partial\nu}\, \mathrm{d}\sigma-\int_{\Omega_{2,\varepsilon_2}\cap\{\xi^+>0\}} |\nabla \xi^+|^2\, \mathrm{d}V
\end{align*}
Setting
\begin{align*}
    U_{1,\varepsilon_1}=\; &\{ x : d(x,\partial\Gamma_2)=\varepsilon_1 \},
    \quad & U_{2,\varepsilon_2}=\; & \{ x : d(x,\partial\Gamma_2)=\varepsilon_2 \},\\
     G_{\varepsilon_1}=\; & \partial \Omega_{1,\varepsilon_1} \setminus U_{1,\varepsilon_1}, \quad & G_{\varepsilon_2}=\;& \partial \Omega_{2,\varepsilon_2} \setminus U_{2,\varepsilon_2}
\end{align*}
we can write
\begin{equation*}
    \partial \Omega_{1,\varepsilon_1} = U_{1,\varepsilon_1} \cup G_{\varepsilon_1}, \quad 
    \partial \Omega_{2,\varepsilon_2} = U_{2,\varepsilon_2} \cup G_{\varepsilon_2}, 
\end{equation*}
and consequently
\begin{equation*}
    \int_{\partial \Omega_{1,\varepsilon_1}} \xi^+\dfrac{\partial v}{\partial\nu} \, \mathrm{d} \sigma = 
    \int_{\Omega\cap U_{1,\varepsilon_1}} \xi^+\dfrac{\partial v}{\partial\nu} \, \mathrm{d}\sigma +
    \int_{G_{\varepsilon_1}} \xi^+\dfrac{\partial v}{\partial\nu} \, \mathrm{d}\sigma 
\end{equation*}
and 
\begin{equation*}
    \int_{\partial \Omega_{2,\varepsilon_2}} \xi^+\dfrac{\partial v}{\partial\nu} \, \mathrm{d} \sigma = 
    \int_{\Omega\cap U_{2,\varepsilon_2}} \xi^+\dfrac{\partial v}{\partial\nu} \, \mathrm{d} \sigma +
    \int_{G_{\varepsilon_2}} \xi^+\dfrac{\partial v}{\partial\nu} \, \mathrm{d}\sigma 
\end{equation*}
Thus,
\begin{equation}\label{Zir}
    0 \leq  
    \int_{\Omega\cap U_{1,\varepsilon_1}} \xi^+\dfrac{\partial v }{\partial\nu}  \, \mathrm{d} \sigma +
    \int_{\Omega\cap U_{2,\varepsilon_2}} \xi^+\dfrac{\partial v}{\partial\nu} \, \mathrm{d} \sigma +
    \int_{G_{\varepsilon_1}\cup G_{\varepsilon_2}} \xi^+\dfrac{\partial v}{\partial\nu} \, \mathrm{d} \sigma 
    -\int_{(\Omega_{1,\varepsilon_1}\cup \Omega_{2,\varepsilon_2})\cap\{\xi^+>0\}} |\nabla \xi^+|^2  \, \mathrm{d} V
\end{equation}
Note that
\begin{equation*}
    \lim_{\varepsilon_1, \varepsilon_2 \to 0} G_{\varepsilon_1}\cup G_{\varepsilon_2} = \Gamma_0 \cup \Gamma_1 \cup \Gamma_2.
\end{equation*}
\bigskip

\begin{claim} The following limits hold:
    \begin{equation*}
        \begin{aligned}
& \lim_{\varepsilon_i \to 0}\int_{\Omega\cap U_{i,\varepsilon_i}} \xi^+\dfrac{\partial v }{\partial\nu}  \, \mathrm{d} \sigma = 0, \quad i=1,2,    \\
    & \lim_{\varepsilon_1, \varepsilon_2 \to 0} 
    \int_{G_{\varepsilon_1}\cup G_{\varepsilon_2}} \xi^+\dfrac{\partial v}{\partial\nu} \, \mathrm{d} \sigma =
    \int_{\Gamma_0\cup \Gamma_1 \cup \Gamma_2} \xi^+\dfrac{\partial v}{\partial\nu} \, \mathrm{d} \sigma \\
  & \lim_{\varepsilon_1, \varepsilon_2 \to 0}  \int_{(\Omega_{1,\varepsilon_1}\cup \Omega_{2,\varepsilon_2})\cap\{\xi^+>0\}} |\nabla \xi^+|^2  \, \mathrm{d} V =
   \int_{\Omega} |\nabla \xi^+|^2  \, \mathrm{d} V
        \end{aligned}
    \end{equation*}
\end{claim}
\noindent Indeed, note that $f=\xi^+ |\nabla \xi | \in L^2(\Omega)$. Thus, the functions $g_i:(0,1) \to (0,\infty)$ defined by
\begin{equation*}
    g_i(\rho)= \int_{U_{i,\rho}} f \, \mathrm{d} \sigma, \quad i=1,2,
\end{equation*}
belongs to $L^1(0,1)$, in view of the coarea formula. 
Moreover, for all $\delta_j = 1/j$, with $j \in \mathds{N}$, there exists $\varepsilon^j_1, \, \varepsilon^j_2 \in (0, \delta_j)$, which we can assume to be monotone decreasing, such that 
\begin{equation*}
    0 \leq \int_{\Omega\cap U_i,\varepsilon^j_i} f \,  \mathrm{d} \sigma
    \leq \dfrac{1}{\delta_j}\int _0^{\delta_j}\int_{\Omega\cap  U_i,\rho} f \,  \mathrm{d} \sigma \mathrm{d} \rho = \dfrac{1}{\delta_j}\int_{\Omega\cap  \left\{ x : \mathrm{d}(x, \partial \Gamma_l) <\delta_j \right\}} f  \mathrm{d} V.
\end{equation*}
Using H\"{o}lder inequality we have 
\begin{eqnarray*}
    \dfrac{1}{\delta_j} \int_{\Omega \cap  \left\{ x  : \mathrm{d}(x, \partial \Gamma_l) < \delta_j\right\}} f \,  \mathrm{d} V &\leq &  
    \left(\dfrac{|\left\{ x \in \Omega : \mathrm{d}(x, \partial \Gamma_l) < \delta_j \right\}|}{\delta_j^2}\right)^{1/2} \left( \int_{\Omega\cap  \left\{ x  : \mathrm{d}(x, \partial \Gamma_l) < \delta_j \right\}} f^2 \, \mathrm{d} V\right)^{1/2} \\
    &\leq & C \left( \int_{\Omega\cap  \left\{ x : \mathrm{d}(x, \partial \Gamma_l) < \delta_j \right\}} f^2 \, \mathrm{d} V\right)^{1/2},
\end{eqnarray*}
where we have used that
\begin{equation*}\label{bounded}
    |\left\{ x : \mathrm{d}(x, \partial \Gamma_l) < \varepsilon \right\}| \leq C \varepsilon^2 , \qquad l=0,2,
\end{equation*}
since $\partial\Gamma_0$ (respectively $\partial\Gamma_2$) is a smooth $n-2$-dimensional manifold and $\partial \Omega_2 \setminus \Gamma_0 $ (respectively $\partial\Omega_0\setminus\Gamma_2$) is smooth enough outside $\partial\Gamma_0$ (respectively $\partial\Gamma_2$), for details see  \cite{federerGMT}.

 Since $f \in L^2(\Omega)$, the right-hand side converges to $0$ as $\delta_j=\dfrac{1}{j}\to 0$, hence
\begin{equation*}
  \int_{\Omega\cap U_i,\varepsilon^j_i} f \,  \mathrm{d} \sigma \to 0,
\end{equation*}
which implies that 
\begin{equation*}
    \int_{{U_i,\varepsilon^j_i}} \xi^+ \frac{\partial \xi}{\partial \nu } \, \mathrm{d} \sigma  \to 0
\end{equation*}
because
\begin{equation*}
    \left|  \int_{{U_i,\varepsilon^j_i}} \xi^+ \frac{\partial \xi}{\partial \nu } \, \mathrm{d} \sigma  \right| \leq 
     \int_{{U_i,\varepsilon^j_i}}  \xi^+ |\langle \nabla \xi , \nu \rangle | \, \mathrm{d} \sigma \leq 
       \int_{{U_i,\varepsilon^j_i}}  \xi^+ | \nabla \xi  | \, \mathrm{d} \sigma =  \int_{{U_i,\varepsilon^j_i}} f \, \mathrm{d} \sigma.
\end{equation*}
Next, we consider the sequence of $L^1$-functions 
\begin{equation*}
    \phi_j  = \xi^+ \langle \nabla h , \nu \rangle  \chi_{G_{\varepsilon^j_1} \cup G_{\varepsilon^j_2}}.
\end{equation*}
Since
\begin{equation*}
    \phi_j \to \xi^+ \langle \nabla \xi , \nu \rangle  \chi_{\Gamma_0 \cup \Gamma_1 \cup \Gamma_2}.
\end{equation*}
almost everywhere in $\Omega$ and $\phi_j \leq h^+ |\nabla \xi|$ for all $j \in  \mathds{N}$ by using the Lebesgue theorem we have 
\begin{equation*}
    \int_{G_{\varepsilon^j_1} \cup G_{\varepsilon^j_2}} \xi^+ \frac{\partial \xi}{\partial \nu } \, \mathrm{d} \sigma   \to \int_{\Gamma_0\cup \Gamma_1 \cup \Gamma_2}  \xi^+ \frac{\partial \xi}{\partial \nu } \, \mathrm{d} \sigma  .
\end{equation*}
Since 
\begin{equation*}
   \int_{\Gamma_0 \cup \Gamma_1 \cup \Gamma_2 }  \xi^+ \frac{\partial \xi}{\partial \nu } \, \mathrm{d} \sigma =
     \int_{\Gamma_0 }  \xi^+ \frac{\partial \xi}{\partial \nu } \, \mathrm{d} \sigma +
       \int_{\Gamma_1 } \xi^+ \frac{\partial \xi}{\partial \nu } \, \mathrm{d} \sigma +
         \int_{ \Gamma_2 }  \xi^+ \frac{\partial \xi}{\partial \nu } \, \mathrm{d} \sigma   \leq 0.
\end{equation*}
Thus taking the limit in \eqref{Zir}, we obtain 
\begin{equation*}
 0 \leq   
 \int_{\Gamma_0 \cup \Gamma_1 \cup \Gamma_2 }  \xi^+ \frac{\partial \xi}{\partial \nu } \, \mathrm{d} \sigma 
 -  \int_{\Gamma_0 } |\nabla \xi^+|^2  \, \mathrm{d} V \leq 0,
\end{equation*}
which implies  that 
\begin{equation*}
    \int_{\Omega } |\nabla \xi^+|^2  \, \mathrm{d} V = 0.
\end{equation*}
Since $ \xi^+ =0 $ in $\Gamma_1 \cup \Gamma_2$, we conclude that $\xi^+ \equiv 0$ in $\Omega,$ and consequently $v \leq c_3$.
\end{proof}

\bigskip 

Now, let us provide a divergence theorem for annular sector-like domains as follows.

\begin{theorem} \label{DT}

Let $X$ be a vector field over $\Omega$ such that
\begin{equation*}
    |X|\in C^1(\Omega\cup \Gamma_0\cup \Gamma_1\cup \Gamma_2)\cap {L^2(\Omega) }\qquad and \qquad \operatorname{div}(X)\in L^1(\Omega).
\end{equation*}
 Then
\begin{equation*}
    \int_\Omega \operatorname{div}(X)\mathrm{d}V=\int_{\partial\Omega} \langle X,\nu\rangle\mathrm{d}\sigma=\int_{\Gamma_0} \langle X,\nu\rangle\mathrm{d}\sigma+\int_{\Gamma_1} \langle X,\nu\rangle\mathrm{d}\sigma+\int_{\Gamma_2} \langle X,\nu\rangle\mathrm{d}\sigma.
\end{equation*}
\end{theorem}

\begin{proof} Using the same notation as in Proposition \ref{MP}, we see that
\begin{eqnarray*}
    \int_{\Omega_{1,\varepsilon_1}} \operatorname{div}(X) \, \mathrm{d} V + \int_{\Omega_{2,\varepsilon_2}} \operatorname{div}(X) \, \mathrm{d}V
      &=&\int_{\partial \Omega_{1,\varepsilon_1}} \langle X,\nu\rangle\, \mathrm{d}\sigma+
    \int_{\partial \Omega_{2,\varepsilon_2}} \langle X,\nu\rangle\, \mathrm{d}\sigma\\
    &=& \int_{\Omega\cap U_{1,\varepsilon_1}} \langle X,\nu\rangle \, \mathrm{d} \sigma +
    \int_{\Omega \cap U_{2,\varepsilon_2}} \langle X,\nu\rangle \, \mathrm{d} \sigma +
    \int_{G_{\varepsilon_1}\cup G_{\varepsilon_2}} \dfrac{\partial v}{\partial\nu} \, \mathrm{d} \sigma.
\end{eqnarray*}
By setting $f=|X|$ and $\phi_j=\langle X,\nu\rangle \chi_{G_{\varepsilon^j_1} \cup G_{\varepsilon^j_2}}$, it follows from an argument similar to the previous proof that 
\begin{equation*}
    \int_\Omega \operatorname{div}(X)\mathrm{d}V=\int_{\partial\Omega} \langle X,\nu\rangle\mathrm{d}\sigma=\int_{\Gamma_0} \langle X,\nu\rangle\mathrm{d}\sigma+\int_{\Gamma_1} \langle X,\nu\rangle\mathrm{d}\sigma+\int_{\Gamma_2} \langle X,\nu\rangle\mathrm{d}\sigma.
\end{equation*}.
\end{proof}

    \section{Some rigidity results in space forms}\label{P-funcao} 
This section is devoted to the study of rigidity phenomena for Serrin-type problems in annular sector-like domains contained within space forms. For the reader’s convenience, we begin by presenting an Obata-type rigidity result for such domains in a space form.
\begin{lemma}\label{Obata}
      Let $\Omega$ be an annular sector-like domain contained in a space form with sectional curvature $k$. Assume that there exists a function
   $u: \overline{\Omega} \to \mathbb{R}$ such that $u \in C^2(\Omega)\cap C^1(\overline{\Omega})$  and $u$ is a solution to
\begin{equation}\label{sanalo}
    \begin{cases}
        \nabla^2 u = -(1 + ku)g & \text{in } \Omega, \\
        u = a_0 & \text{on } \Gamma_0  \\
        u = a_2 & \text{on } \Gamma_2,
    \end{cases}
\end{equation}
where $ a_0, a_2 \in \mathbb{R}$. Then $\Gamma_0=\partial B_{R_1}(\mathcal{O})\cap \Sigma$ and $\Gamma_2=\partial B_R(\mathcal{O})\cap \Sigma$, that is, \( \Omega = \{x\in \Sigma: R< d(x,\mathcal{O})< R_1\} \) and $u$ is a radial function, that is, $u$ depends only on the distance from the center of the ball $B_R(\mathcal{O})$. Moreover, we can describe explicitly the function $u$ for each case as follows:
\begin{flushleft}
$ \mathrm{(i)}$ For the Euclidean space under the overdetermined boundary condition described in Theorem~\ref{Euclidean space}, we have 
\begin{equation*}
    u(x) = \frac{R_1^2}{2} - \frac{r^2}{2}, \quad \text{where} \quad  r=\operatorname{dist}(x,\mathcal{O}).
\end{equation*} 
$ \mathrm{(ii)}$ For the elliptic space under the overdetermined boundary condition described in Theorem~\ref{cajarana} or Theorem~\ref{general case}, we have
\begin{equation*}
    u(x) = \frac{\cos r -\cos R_1}{\cos R_1}, \quad \text{where} \quad  r=\operatorname{dist}(x,\mathcal{O}).
\end{equation*}  
$ \mathrm{(iii)}$ For the hyperbolic space under the overdetermined boundary condition described in Theorem~\ref{jenipapo} or Theorem~\ref{jenipapo2}, we have
\begin{equation*}
    u(x) = \frac{\cosh R_1 -\cosh r}{\cosh R_1}, \quad \text{where} \quad  r=\operatorname{dist}(x,\mathcal{O}).
\end{equation*}    
\end{flushleft}

\end{lemma}
\begin{proof}
    Since $u$ is constant on $\Gamma_0$, we have $\nu=\nabla u/|\nabla u|$ on $\Gamma_0$. Then, for all $X,Y\in \mathfrak X(\Gamma_0)$,
    \begin{align*}
        A(X,Y)=\langle \nabla_X\nu,Y\rangle = \left\langle \nabla_ X \dfrac{\nabla u}{|\nabla u|},Y\right\rangle  =\dfrac{1}{|\nabla u|} \langle \nabla_X\nabla u,Y\rangle.
    \end{align*}
That is, 
    \begin{equation}\label{fijes}
        A(X,Y)=-\dfrac{1+ka_0}{|\nabla u|}\langle X,Y\rangle
    \end{equation}
    Thus, $\Gamma_0$ is totally umbilical in $M$. Analogously, $\Gamma_2$ is totally umbilical in $M$. Therefore, since $M$ is  a space form we have $\Gamma_0=\partial B_{R_1}(\mathcal{O})\cap \Sigma$ and $\Gamma_2=\partial B_R(\mathcal{O})\cap \Sigma$ with $R<R_1$. Consequently, by \eqref{fijes} we have $|\nabla u|=c_0$ (constant) and $\nabla u=c_0\nu$ on  $\Gamma_0$. Analogously, $|\nabla u|=c_2$ (constant) and $\nabla u=-c_2\nu$ on  $\Gamma_2$.

    Now, choose $ \rho_2 \in (R, R_1),\,  q_1, \, q_2 \in \Sigma\cap \partial B_g(\mathcal{O}, \rho_2)\setminus\overline{\Gamma}_1$.
Since $M$ is a space form, for each $i = 1, 2$, there exists a unit minimizing geodesic $\gamma_i: [0, \rho_2] \to \Sigma\cap\overline{B_g(\mathcal{O}, \rho_2)}$ such that
    \begin{equation*}
        \gamma_i(0) = \mathcal{O} \quad \text{and} \quad \gamma_i(\rho_2) = q_i.
    \end{equation*}
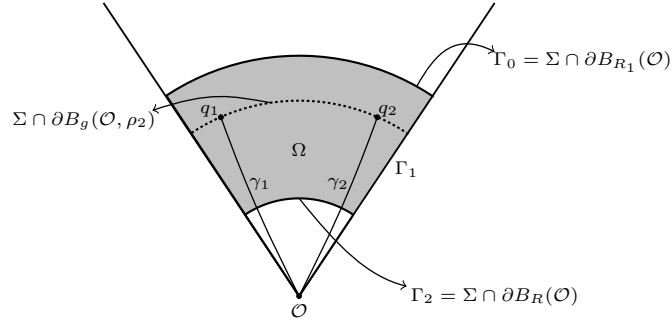
\begin{center}
        \begin{tikzpicture}[scale=1.3]

\draw [color=black,line width=0.6pt, dash pattern=on 1pt off 1pt,thick,domain=56:123.5] plot ({2*cos(\x)}, {2*sin(\x)});

\draw[line width=0.7pt,color=black] (2,3)--(0,0)-- (-2,3);

             \draw [color=black,line width=0.6pt,fill=black,fill opacity=0.25,thick,domain=56:123.5] plot ({2.46*cos(\x)}, {2.46*sin(\x)})--(0,0);

            \draw [color=black,fill=white,fill opacity=1,thick,domain=56:123.5] plot ({cos(\x)}, {sin(\x)})--(0,0)--(0.553,0.83);

            \draw [fill=black] (0,0) circle (0.7pt);
            \draw [fill=black] (-0.8,1.83) circle (0.7pt);
            \draw [fill=black] (0.8,1.83) circle (0.7pt);

            \draw[line width=0.5pt,color=black,->] (0,1) .. controls (0.4,0.5) and (0.7,0.2) ..  (1.1,0.1);

            \draw[line width=0.5pt,color=black,->] (-0.3,1.98) .. controls (-0.7,2.05) and (-0.9,2.1) ..  (-1.5,1.9);
           
            \draw[line width=0.5pt,color=black,->] (1.2,2.14) .. controls (1.4,2.5) and (1.6,2.7) ..  (2,2.5);

            \draw[line width=0.5pt,color=black] (0,0) .. controls (0.3,0.6) and (0.5,1) ..  (0.8,1.83);
          
            \draw[line width=0.5pt,color=black] (0,0) .. controls (-0.3,0.6) and (-0.5,1) ..  (-0.8,1.83);

            \draw  (0,-0.15) node {\tiny \textcolor{black}{$\mathcal{O}$}};

            \draw  (-0.4,1.15) node {\tiny \textcolor{black}{$\gamma_1$}};
            
            \draw  (0.4,1.15) node {\tiny \textcolor{black}{$\gamma_2$}};

            \draw  (1.1,1.3) node {\tiny \textcolor{black}{$\Gamma_1$}};

            \draw  (-0.91,1.9) node {\tiny \textcolor{black}{$q_1$}};
            \draw  (0.91,1.9) node {\tiny \textcolor{black}{$q_2$}};

            \draw  (0,1.5) node {\tiny \textcolor{black}{$\Omega$}};

            \draw  (2,0) node {\tiny \textcolor{black}{$\Gamma_2=\Sigma\cap\partial B_R(\mathcal{O})$}};

            \draw  (2.9,2.4) node {\tiny \textcolor{black}{$\Gamma_0=\Sigma\cap\partial B_{R_1}(\mathcal{O})$}};
            
            \draw  (-2.2,1.8) node {\tiny \textcolor{black}{$\Sigma\cap\partial B_g(\mathcal{O},\rho_2)$}};
        \end{tikzpicture}
    \end{center}
    For each $i = 1, 2$, define $f_i(s) = u(\gamma_i(s))$ on $[R, \rho_2]$. Then,
    \begin{equation*}
        f_i'(s) = \langle \nabla u, \gamma'(s) \rangle, \quad \forall s \in (R, \rho_2).
    \end{equation*}
Since $(\Sigma\cap \overline{B_g(\mathcal{O}, \rho_2)}) \setminus (\overline{B_R(\mathcal{O})}\cup \Gamma_1) \subset \Omega$ and $u$ satisfies \eqref{sanalo}, we have for all  $s \in (R, \rho_2),$
    \begin{eqnarray*}
        f_i''(s) &=& \langle \nabla_{\gamma_i'(s)} \nabla u, \gamma_i'(s) \rangle + \langle \nabla u, \nabla_{\gamma_i'(s)} \gamma_i'(s) \rangle \\
        &=& \nabla^2 u(\gamma_i'(s), \gamma_i'(s)) \\
        &=& -(1 + k u(\gamma_i(s))) \cdot |\gamma_i'(s)|^2 \\
        &=& -1 - k f_i(s).
    \end{eqnarray*}
    By the boundary condition on $\Gamma_2$,
    \begin{equation*}
        f_i(R) = a_2 \quad \text{and} \quad f_i'(R) = -c_2.
    \end{equation*}
Thus, $f_i(s)$ satisfies  the following initial value problem (IVP):
    \begin{equation}\label{aire}
        \begin{cases}
            v''(s) + k v(s) = -1 & \text{in } (R, \rho_2), \\
            v'(R) = -c_2, \\
            v(R) = a_2.
        \end{cases}
    \end{equation}
Using the regularity and uniqueness properties of Problem \eqref{aire}, we conclude
\[
f_1 \equiv f_2 \quad \text{in} \quad [R, \rho_2].
\]
Therefore, \(u\) has the same expression along \(\gamma_1|_{[R,\rho_2]}\) and \(\gamma_2|_{[R,\rho_2]}\); consequently, \(u\) is constant on 
$
\Sigma \cap \partial B_g(\mathcal{O},\rho_2) \setminus \Gamma_1, 
$
that is, for each \(\rho_2 \in (R,R_1)\) we have $u(x)=u(y)$ if $x,y \in \Sigma \setminus \Gamma_1$ and $\operatorname{d}(\mathcal{O},x)=\operatorname{d}(\mathcal{O},y)=\rho_2$. 
By continuity, the same holds when including the boundary $\Gamma_1$, precisely on \(\Sigma \cap \partial B_g(\mathcal{O},\rho_2)\) and by the same argument, we also have the result on \(\Sigma \cap \partial B_{R_1}(\mathcal{O})\) and \(\Sigma \cap \partial B_{R}(\mathcal{O})\). 

To complete the proof of Lemma~\ref{Obata}, we just use the uniqueness of solutions for IVP \eqref{aire}
\end{proof}

\begin{proposition}
    Let \( N \) be a hypersurface of \( M \). If \( f \in C^2(M) \), then
\[
\Delta f = \Delta^N f + \langle \nabla f, \nu \rangle H + \nabla^2 f(\nu, \nu),
\]
where \( \Delta^N \) is the Laplacian of \( N \) and \( \nu \) is the outward unit normal vector.
\end{proposition}

    \begin{proof}
   Let \(\{e_i\}_{i=1}^n\) be a local orthonormal frame such that \(e_n = \nu\). By definition,
   \begin{equation*}
       \Delta f = \sum_{i=1}^n \langle \nabla_{e_i} \nabla f, e_i \rangle = \sum_{i=1}^{n-1} \langle \nabla_{e_i} \nabla f, e_i \rangle + \langle \nabla_{\nu} \nabla f, \nu \rangle.
   \end{equation*}
   Let \(\nabla^N\) denote the gradient of \(f\) on \(N\) and recall that the mean curvature is defined by $H=\operatorname{tr} A.$ Thus,
   \begin{align*}
        \Delta f &= \sum_{i=1}^{n-1} \langle \nabla_{e_i} (\nabla^N f + \langle \nabla f, \nu \rangle \nu), e_i \rangle + \nabla^2 f(\nu, \nu) \\
        &= \sum_{i=1}^{n-1} \langle \nabla_{e_i} \nabla^N f, e_i \rangle + \langle \nabla f, \nu \rangle \sum_{i=1}^{n-1} \langle \nabla_{e_i} \nu, e_i \rangle + \nabla^2 f(\nu, \nu) \\
        &= \Delta^N f + \langle \nabla f, \nu \rangle H + \nabla^2 f(\nu, \nu).
   \end{align*}
\end{proof}

\subsection{Proof of Theorem~\ref{Euclidean space}}  
Let $\Omega$ be an {\it annular sector-like domain} in the Euclidean space as in Definition~\ref{def_sec}.   By the Maximum Principle, Proposition~\ref{MP},  $u>0$ in $\Omega$. 
   Consider the auxiliary function $f=u+\dfrac{|X|^2}{2}$, where $X$ is the position vector in $\mathbb{R}^n$ minus $\mathcal{O}$,
   that is, $X(x)=x-\mathcal{O}$, where  $\mathcal{O}$ is the vertex of the cone. From the polarized Bochner formula we have
\begin{equation*}
    \Delta\langle \nabla f,\nabla u\rangle=\langle \nabla (\Delta f), \nabla u\rangle+\langle \nabla f,\nabla(\Delta u)\rangle +2\operatorname{tr}(\nabla^2 f\circ \nabla^2 u)+2Ric(\nabla f, \nabla u).
\end{equation*}
   Note that  
\[
\Delta f = \Delta u + \dfrac{\Delta |X|^2}{2} = \Delta u + n = 0
\]
and  
\begin{align*}
\operatorname{tr}\left(\nabla^2 f \circ \nabla^2 u\right) &= \operatorname{tr}\left(\nabla^2 \left(u + \dfrac{|X|^2}{2}\right) \circ \nabla^2 u\right) \\
&= |\nabla^2 u|^2 + \dfrac{1}{2}\langle \nabla^2 |X|^2, \nabla^2 u \rangle \\
&= |\nabla^2 u|^2 + \Delta u = |\nabla^2 u|^2 - \dfrac{1}{n}(\Delta u)^2 \geq 0.
\end{align*}
Thus,  
\begin{equation}\label{fan}
    \Delta\langle \nabla f, \nabla u \rangle = 2\operatorname{tr}\left(\nabla^2 f \circ \nabla^2 u\right) \geq 0,
\end{equation}
with equality holding if and only if the Hessian of $u$ is multiple of the metric of $M$, more precisely, $\nabla^2 u =\frac{\Delta u}{n} g = -g$. Since $u$ is a solution of \eqref{saber}, 
    \begin{eqnarray*}
    \frac{\partial f}{\partial \nu}
        = \frac{\partial u}{\partial \nu} + \frac{\partial }{\partial \nu} \left(\dfrac{|X|^2}{2}\right)=\dfrac{\partial u}{\partial \nu}+\langle \nu, X \rangle=0 \qquad on \ \Gamma_1,
    \end{eqnarray*}
because the field $X$ is tangent to the boundary $\Gamma_1$. Also,  
      \begin{eqnarray*}
         \dfrac{\partial f}{\partial \nu}=\dfrac{\partial u}{\partial \nu}+\dfrac{\partial}{\partial \nu}\left(\dfrac{|X|^2}{2}\right)=R+\langle \nu,X\rangle=R-R=0 \qquad on \ \Gamma_2,
    \end{eqnarray*}
 where we have used that $\nu=-{X}/{|X|}$ on $\Gamma_2$.

By the Divergence Theorem, Proposition~\ref{DT}, we obtain
\begin{align*}
   \int_\Omega(u\Delta \langle \nabla f,\nabla u\rangle -\langle \nabla f, \nabla u\rangle \Delta u) 
&= \int_\Omega\operatorname{div}(u\nabla\langle \nabla f,\nabla u\rangle-\langle \nabla f,\nabla u\rangle \nabla u)\nonumber \\
&= \int_{\partial\Omega}u\langle \nabla \langle \nabla f,\nabla u\rangle,\nu\rangle-\int_{\partial\Omega}\langle \nabla f,\nabla u\rangle \dfrac{\partial u}{\partial \nu}\nonumber \\
&=  \int_{\Gamma_1} u\langle \nabla \langle \nabla f,\nabla u\rangle,\nu\rangle+ \int_{\Gamma_2} u\langle \nabla \langle \nabla f,\nabla u\rangle,\nu\rangle 
-\int_{\Gamma_0}\langle \nabla f,\nabla u\rangle \dfrac{\partial u}{\partial \nu},
\end{align*}

which gives

\begin{equation}\label{rosas}
\int_\Omega(u\Delta \langle \nabla f,\nabla u\rangle -\langle \nabla f, \nabla u\rangle \Delta u) \\
= \int_{\Gamma_1} u\langle \nabla \langle \nabla f,\nabla u\rangle,\nu\rangle+ \int_{\Gamma_2} u\langle \nabla \langle \nabla f,\nabla u\rangle,\nu\rangle 
        -\int_{\Gamma_0}\langle \nabla f,\nabla u\rangle \dfrac{\partial u}{\partial \nu}.
\end{equation}

We also have
\begin{equation}\label{otro}
    \begin{aligned}
    \int_\Omega \langle \nabla f,\nabla u\rangle \Delta u&=n\int_{\Omega} \langle \nabla f,\nabla u\rangle \\
  &=n\int_{\Omega}\operatorname{div}(u\nabla f)-n\int_\Omega u\Delta f \\
    &=n\int_{\partial\Omega}u \left\langle \nabla f,\nu\right\rangle   .     
    \end{aligned}
\end{equation}
We recall that $\nabla u=u_\nu \nu=-R_1\nu$ on $\Gamma_0$, hence  
\begin{equation}\label{habra}
\begin{aligned}
          \int_{\Gamma_0}\langle \nabla f,\nabla u\rangle = & 
          -R_1\int_{\Gamma_0}\langle \nabla f,\nu\rangle-R_1\int_{\Gamma_1}\langle \nabla f,\nu\rangle-R_1\int_{\Gamma_2}\langle\nabla f,\nu\rangle \\
          = & -R_1\int_{\partial\Omega}\langle\nabla f,\nu\rangle=-R_1\int_\Omega \Delta f = 0.
\end{aligned}
\end{equation}

Let $\Delta^{\Gamma_2}$ be the Laplacian of $\Gamma_2$ and $H$ the mean curvature of $\Gamma_2$.
We recall that
\begin{equation*}
    \Delta u=\Delta^{\Gamma_2} u+H\langle\nabla u,\nu\rangle+\nabla^2u(\nu,\nu),
\end{equation*}
along  $\Gamma_2$. Since $u$ is constant on $\Gamma_2$, we have $\Delta^{\Gamma_2}u=0$. Since  $\Gamma_2$ is a ball of radius $R$,  the mean curvature of $\Gamma_2$ is equal to $-(n-1)/R$ and, therefore, 
\begin{equation*}
   -n=R H+\nabla^2u(\nu,\nu)=-(n-1)+\nabla^2u(\nu,\nu),
\end{equation*}
which implies that $\nabla^2u(\nu,\nu)=-1$. 

Now, note that on $\Gamma_2$ we have
\begin{equation}\label{a10}
\begin{aligned}
           \dfrac{\partial}{\partial \nu}  \langle \nabla f,\nabla u\rangle&= \langle \nabla_\nu \nabla f,\nabla u\rangle +\langle \nabla f,\nabla_\nu \nabla u\rangle \\
    &= \nabla^2 f(\nu, \nabla u)+\nabla^2 u(\nu, \nabla f)\\
    &= 2\nabla^2 u(\nu, \nabla u)+\dfrac{1}{2}\nabla^2|X|^2 (\nu,\nabla u)+\nabla^2 u\left(\nu , \dfrac{\nabla |X|^2}{2}\right)\\
    &= R\nabla^2 u(\nu, \nu)+ \langle \nu,\nabla u\rangle\\
    &= R(\nabla^2 u(\nu, \nu)+1) =0. 
\end{aligned}
\end{equation}

 Next, we analyze the following term
    \begin{equation*}
        \int_{\Gamma_1} u\langle \nabla \langle \nabla f,\nabla u\rangle,\nu\rangle.
    \end{equation*}
    Note that, on $\Gamma_1$, 
    \begin{equation*}
        \dfrac{\partial u}{\partial \nu}=0 \quad and \quad \nabla_X\nu =0.
    \end{equation*}
 Thus, 
\begin{eqnarray*}
    0=X\left(\dfrac{\partial u}{\partial \nu}\right)&=& X\langle \nabla u, \nu\rangle\\
    &=& \langle \nabla_X\nabla u, \nu\rangle +\langle \nabla u, \nabla_X \nu\rangle\\\
    &=& \nabla^2 u(X,\nu).
\end{eqnarray*}
Moreover, from the convexity of the cone, we obtain  
\begin{equation*}
    A(\nabla u,\nabla u)=\langle\nabla_{\nabla u}\nu, \nabla u\rangle\geq 0 \quad on \ \Gamma_1,
\end{equation*}
Therefore, on $ \Gamma_1$,
\begin{eqnarray*}
    0=\nabla u\left(\dfrac{\partial u}{\partial \nu }\right)&=& \nabla u\langle \nabla u,\nu\rangle \\
    &=& \langle \nabla_{\nabla u}\nabla, \nu\rangle+\langle \nabla u, \nabla_{\nabla u}\nu\rangle\\ 
    &=& \nabla^2 u(\nabla u,\nu )+A(\nabla u, \nabla u)\\
    &\geq & \nabla^2u(\nabla u, \nu).
\end{eqnarray*}
Note also that
\begin{equation*}
    \nabla f=\nabla u+X,
\end{equation*}
hence
\begin{eqnarray*}
    \langle \nabla \langle \nabla f,\nabla u\rangle,\nu\rangle &=& \langle \nabla \langle \nabla u + X,\nabla u\rangle,\nu\rangle \\
    &=& \langle \nabla |\nabla u|^2,\nu\rangle+\nu \langle X, \nabla u\rangle\\
    &=& 2\langle \nabla_\nu \nabla u,\nabla u\rangle +\langle \nabla_\nu  X, \nabla u\rangle + \langle X, \nabla_\nu \nabla u\rangle\\
    &=&2 \nabla^2 u(\nabla u, \nu)+\nabla^2 u(X,\nu).
\end{eqnarray*}
In this way, since $u>0$, we conclude that
\begin{equation}\label{bucle}
     u\langle \nabla \langle \nabla f,\nabla u\rangle,\nu\rangle \leq 0.
\end{equation}
   Substituting \eqref{otro},  \eqref{habra}, \eqref{a10} and \eqref{bucle} into \eqref{rosas},
    \begin{equation*}
        \int_\Omega u\Delta \langle \nabla f,\nabla u\rangle \leq 0,
    \end{equation*}
    but $u\Delta\langle \nabla f,\nabla u\rangle\geq 0$, consequently,  $\Delta\langle \nabla f,\nabla u\rangle= 0$ in $\Omega$, that is, equality holds in  \eqref{fan}. Therefore, 
    \begin{equation*}
        \nabla^2u=-g.
    \end{equation*}
   Therefore, by Lemma~\ref{Obata}, the result follows. Moreover,
   \begin{equation*}
       u(x)=\dfrac{R_1^2}{2}-\dfrac{r^2}{2}
   \end{equation*}
   where $r=\operatorname{dist}(x,p)$.

\subsection{Spherical case} 

Next, we provide the proofs for two rigidity results addressed to annular sector-like domains contained in the sphere. 
\bigskip

The following result was established in \cite{lee2023radial}. For completeness, we provide a brief outline of the proof here.

\begin{proposition}\label{caja} Let $u$ be a solution of the Serrin problem \eqref{saber2}. The $P-$function $P:=|\nabla u|^2 + 2 u + u^2$ satisfies the following properties,  
        \begin{equation}
        \begin{cases}
            \Delta P \geq 0 & in \ \Omega\\
            \dfrac{\partial P}{\partial \nu } \leq 0 & on \ \Gamma_1\\
            P= \frac{\sin^2 R_1}{\cos^2 R_1} & on \ \Gamma_0\\
             P= \frac{\sin^2 R}{\cos^2 R_1} + 2a + a^2 & on \ \Gamma_2.\\
        \end{cases}
    \end{equation}
Moreover, 
\begin{equation*}
    \Delta P =0 \Longleftrightarrow \nabla^2 u = -(1+u)g.
\end{equation*}
\end{proposition}

\begin{proof}
    It is known that

$$(\Delta u)^2 \leq n |\nabla^2 u|^2.$$
Equality holds if and only if $\nabla^2 u$ is proportional to the metric $g$. Applying the Bochner formula, we have:
\begin{align*} 
\Delta |\nabla u|^2 &= 2 \langle \nabla(\Delta u), \nabla u \rangle + 2 |\nabla^2 u|^2 + 2 \text{Ric}(\nabla u, \nabla u) \\ 
&\geq -2n |\nabla u|^2 + \frac{2}{n}(\Delta u)^2 + 2(n-1)|\nabla u|^2  \\  
&= -2\Delta u-\Delta u^2.
\end{align*} 
 We  conclude that 

$$\Delta P \geq 0,$$
indicating that $P(u)$ is subharmonic in $\Omega$. Let us recall that $A$ denotes  the second fundamental form and, therefore, using the convexity of the cone, we obtain  
\[
A(\nabla u, \nabla u) = \langle \nabla_{\nabla u} \nu, \nabla u \rangle \geq 0 \quad \text{on } \Gamma_1.
\]
Therefore, on \(\Gamma_1\),
\begin{align*}
    0 = \nabla u\left(\dfrac{\partial u}{\partial \nu}\right) &= \nabla u \langle \nabla u, \nu \rangle \\
    &= \langle \nabla_{\nabla u} \nabla u, \nu \rangle + \langle \nabla u, \nabla_{\nabla u} \nu \rangle \\ 
    &= \nabla^2 u(\nabla u, \nu) + A(\nabla u, \nabla u) \\
    &\geq \nabla^2 u(\nabla u, \nu).
\end{align*}
Thus, from above inequalities
\[
\dfrac{\partial P}{\partial \nu} = 2\nabla^2 u(\nabla u, \nu) + 2\dfrac{\partial u}{\partial \nu} - 2u\dfrac{\partial u}{\partial \nu} \leq 0 \quad \text{on } \Gamma_1.
\]
\end{proof}

\begin{proposition}\label{max} The $P-$function archives its maximum on $\Gamma_0$.
    
\end{proposition}
\begin{proof} Using Propostion~\ref{caja}, we have that $P$ satisfies the conditions of
the Maximum Principle, Proposition~\ref{MP}, and consequently 
        \begin{equation*}
        \max_{x\in \overline{\Omega}} P(x) = \max\left\{\frac{\sin^2 R_1}{\cos^2 R_1},\frac{\sin^2 R}{\cos^2 R_1} + 2a + a^2\right \}
    \end{equation*}
which, together with the next claim, gives the desired result. 
\begin{claim}\label{mangaba} Note that
    \begin{equation}\label{araca}
        \max\left\{\frac{\sin^2 R_1}{\cos^2 R_1},\frac{\sin^2 R}{\cos^2 R_1} + 2a + a^2\right \} =\frac{\sin^2 R_1}{\cos^2 R_1}.
    \end{equation}
\end{claim}
\noindent If $P$ is constant in $\overline{\Omega}$ there is nothing to be proved. Suppose that \eqref{araca} does not hold, that is, 
\begin{equation}\label{a}
    \dfrac{\sin^2 R}{\cos^2 R_1} +2a+a^2> \dfrac{\sin^2 R_1}{\cos^2 R_1}.
\end{equation}
Then,
\begin{eqnarray*}
    (a+1)^2&>& \dfrac{\sin^2 R_1-\sin^2 R}{\cos^2 R_1}+1\\
           &=& \dfrac{1-\cos^2 R_1-1+\cos^2 R+\cos^2 R_1}{\cos^2 R_1}\\
           &=& \dfrac{\cos^2 R}{\cos^2 R_1},
\end{eqnarray*}
which together with the fact that  $a<-1$ gives
\begin{equation}\label{b}
    a+1<\dfrac{\cos R}{\cos R_1}.
\end{equation}
Since $P$ is not constant in $\overline{\Omega}$, by \eqref{a}, using Hopf's lemma we have
\begin{equation}\label{umbu}
    \dfrac{\partial P}{\partial \nu} >0 \quad on \ \Gamma_2.
\end{equation}

Let $\{e_1,\ldots,e_{n-1},\partial_r\}$ a local orthonormal frame of $\Omega$ at $\Gamma_2$, where $\{e_1,\ldots,e_{n-1}\}$ is tangent to $\Gamma_2$. 
Since $u$ is constant on $\Gamma_2$, we have
\begin{equation*}
    u_i=0 \quad and \quad u_{ij}=0 \quad on \ \Gamma_2,
\end{equation*}
for all $i,j=1,\ldots,n-1$. Moreover, since 
$u_r$ is constant on $\Gamma_2$, 
\begin{equation*}\label{c}
    u_{r i}=0 \quad on \ \Gamma_2,
\end{equation*}
for all $i=1,\ldots,n-1$. Then, for all $W,Z\in \mathfrak X(\Omega)$, 
\begin{equation}\label{d}
    \nabla^2 u(W,Z)=\langle W, \partial_r \rangle \langle Z , \partial_r\rangle u_{rr} \quad on \ \Gamma_2 .
\end{equation}
Analogously, 
\begin{equation}\label{d*}
    \nabla^2 u(W,Z)=\langle W, \nu \rangle \langle Z , \nu\rangle u_{\nu \nu} \quad on \ \Gamma_0 ,
\end{equation}
where $\nu$ is orthogonal to $\Gamma_0$.

Let $\Delta^{\Gamma_2}$ be the Laplacian of $\Gamma_2$ and $H$ the mean curvature of $\Gamma_2$.
We recall that
\begin{equation*}
    \Delta u=\Delta^{\Gamma_2} u+H\langle\nabla u,\nu\rangle+\nabla^2u(\nu,\nu),
\end{equation*}
along  $\Gamma_2$.
Since $u$ is constant on $\Gamma_2$, we have $\Delta^{\Gamma_2}u=0$. Since  $\Gamma_2$ is a ball of radius $R$,  the mean curvature of $\Gamma_2$ is equal to $-(n-1){\cos R}/{\sin R}$ and, therefore, 
 \begin{eqnarray*}
     \Delta u&=& \Delta^{\Gamma_2} u-H\langle \nabla u, \partial_r \rangle + \nabla^2 u(\partial_r, \partial_r)\\
            &=& (n-1)\dfrac{\cos R}{\sin R}u_r +u_{rr}, \quad on \ \Gamma_2.
 \end{eqnarray*}

 Thus, 
 \begin{equation*}
     -n-na=(n-1)\dfrac{\cos R}{\sin R}\left(-\dfrac{\sin R}{\cos R_1}\right)+u_{rr},
 \end{equation*}
that is,
\begin{equation}\label{c2}
    u_{rr}=-n-na+(n-1)\dfrac{cos R}{\cos R_1} \quad on \ \Gamma_2.
\end{equation}
By \eqref{d} and \eqref{c2}, the normal derivative of \(P\) on \(\Gamma_2\) is given by
\begin{eqnarray}\label{e}
    \dfrac{\partial P}{\partial \nu}&=& 2\nabla^2 u(\nabla u,\nu)+2\dfrac{\partial u}{\partial \nu}+2u\dfrac{\partial u}{\partial \nu}\nonumber\\
                                    &=& 2u_{rr}\dfrac{\partial u}{\partial \nu}+2\dfrac{\partial u}{\partial \nu}+2u\dfrac{\partial u}{\partial \nu}\nonumber\\
                                    &=& 2(n-1)\dfrac{\sin R}{\cos R_1}\left ( \dfrac{\cos R}{\cos R_1}-1-a\right ) \quad on \ \Gamma_2,
\end{eqnarray}
which together with \eqref{b} implies 
\begin{equation*}
    \dfrac{\partial P}{\partial \nu}<0 \quad on \ \Gamma_2,
\end{equation*}
which is a contradiction with \eqref{umbu}. Therefore, \eqref{araca} holds.
\end{proof}

\begin{proposition}\label{z} The $P-$function is constant in $\overline{\Omega}$
\end{proposition}
\begin{proof} Assume by contradiction that the $P-$function is not constant. 
By \eqref{araca} and \eqref{e}, 
\begin{equation*}
    \dfrac{\partial P}{\partial \nu} \geq 0 \quad on \ \Gamma_2.
\end{equation*}
By applying the Hopf lemma to \(\Gamma_0\) and using \eqref{d*}, we derive
\begin{equation}\label{estrela}
    0<\dfrac{\partial P}{\partial \nu}= 2\nabla^2 u(\nabla u,\nu)+2\dfrac{\partial u}{\partial \nu}+2u\dfrac{\partial u}{\partial \nu}=2\dfrac{\partial u}{\partial \nu}
    (u_{\nu\nu}+1).
\end{equation}
Thus,  $(u_{\nu\nu}+1)>0$ on $\Gamma_0$,  since $\dfrac{\partial u}{\partial \nu}>0$ on $\Gamma_0$. 

Analogously, on $\Gamma_2$ we have
\begin{eqnarray*}    0\leq \dfrac{\partial P}{\partial \nu}= 2\nabla^2 u(\nabla u,\nu)+2\dfrac{\partial u}{\partial \nu}+2u\dfrac{\partial u}{\partial \nu}=2\dfrac{\partial u}{\partial \nu}
    (u_{rr}+1+a).
\end{eqnarray*}
Thus, $(u_{rr}+1+a)\leq 0$ on $\Gamma_0$, since $\dfrac{\partial u}{\partial \nu}<0$ on $\Gamma_2$.

\begin{claim}
    The function $\widetilde{P} \coloneqq \langle \nabla u, \nabla \psi' \rangle + u \psi'+\psi'$ is harmonic in $\Omega$, 
    where $\psi=\sin(r)$.
\end{claim}

To prove this claim, we begin by noting that
$$
\nabla^2 \psi' = -\psi'g, \quad \Delta \psi' = -n\psi',
$$
where $g$ denotes the metric of $S^n$. Using the polarized Bochner formula, we derive
\begin{align*}
\Delta \langle \nabla u, \nabla \psi' \rangle &= \langle \nabla (\Delta u), \nabla \psi' \rangle + \langle \nabla u, \nabla (\Delta \psi') \rangle + 2\operatorname{tr}(\nabla^2 u \circ \nabla^2 \psi') + 2\operatorname{Ric}(\nabla u, \nabla \psi') \\
&= -2\psi'(-nu - n) - 2 \langle \nabla u, \nabla \psi' \rangle \\
&= 2nu\psi' + 2n\psi' - 2 \langle \nabla u, \nabla \psi' \rangle.
\end{align*}
Since
$$
\Delta (u\psi') = -2nu\psi' - n\psi' + 2 \langle \nabla u, \nabla \psi' \rangle,
$$
it follows that
$$
\Delta (\langle \nabla u, \nabla \psi' \rangle + u\psi') = n\psi' = -\Delta \psi',
$$
which establishes $\widetilde{P} $ is harmonic in $\Omega$. 

\bigskip 

For the sake of contradiction, we assume that the $\tilde{P}$-function is not constant. The next step is to analyze the sign of the normal derivative of the $\tilde{P}$-function on the boundary of $\Omega$. From this, we arrive at a contradiction.
\begin{equation*}
    \dfrac{\partial \Tilde{P}}{\partial \nu}= \nabla^2 u(\nabla \psi',\nu)+\nabla^2 \psi'(\nabla u,\nu )+u\dfrac{\partial \psi'}{\partial \nu}+\psi'\dfrac{\partial u}{\partial \nu}+\dfrac{\partial \psi'}{\partial \nu}.
\end{equation*}
 Thus, since $\Omega$ is a star-shaped domain with respect to $\mathcal{O}$,   on $\Gamma_0$ we have
\begin{eqnarray*}
    \dfrac{\partial \Tilde{P}}{\partial \nu}&=&u_{\nu \nu}\dfrac{\partial \psi'}{\partial \nu}-\psi'\dfrac{\partial u}{\partial \nu}+u\dfrac{\partial \psi'}{\partial \nu}+\psi'\dfrac{\partial u}{\partial \nu}+\dfrac{\partial \psi'}{\partial \nu}\\
    &=&(u_{\nu\nu}+1)\dfrac{\partial \psi'}{\partial \nu}\\
    &=& -(u_{\nu\nu}+1)\sin r\langle \nabla r, \nu \rangle <0.
\end{eqnarray*}

\medskip
Analogously, on $\Gamma_2$ we obtain
\begin{eqnarray*}
    \dfrac{\partial \Tilde{P}}{\partial \nu}&=&u_{\nu \nu}\dfrac{\partial \psi'}{\partial \nu}-\psi'\dfrac{\partial u}{\partial \nu}+u\dfrac{\partial \psi'}{\partial \nu}+\psi'\dfrac{\partial u}{\partial \nu}+\dfrac{\partial \psi'}{\partial \nu}\\
    &=&(u_{rr}+1+a)\dfrac{\partial \psi'}{\partial \nu}\\
    &=& -(u_{rr}+1+a)\sin r\langle \nabla r, \nu \rangle \leq 0 \quad on \ \Gamma_2,
\end{eqnarray*}
and on $\Gamma_1$, because $\nabla_{\nabla \psi'}\nu=0$ and $\langle \nabla \psi',\nu \rangle =0$, we conclude that
\begin{eqnarray*}
    \dfrac{\partial \Tilde{P}}{\partial \nu}&=&{\nabla \psi'}\left(\dfrac{\partial u}{\partial \nu}\right)-\psi'\dfrac{\partial u}{\partial \nu}+u\dfrac{\partial \psi'}{\partial \nu}+\psi'\dfrac{\partial u}{\partial \nu}+\dfrac{\partial \psi'}{\partial \nu}=0.
\end{eqnarray*}

\bigskip

On the other hand, using Theorem~\ref{DT}, we obtain 
\begin{equation*}
    0=\int_{\Omega}\Delta \Tilde{P} \, \mathrm{d}V=\int_{\partial \Omega}\dfrac{\partial\Tilde{P}}{\partial \nu}\, \mathrm{d}\sigma=\int_{\Gamma_0}\dfrac{\partial\Tilde{P}}{\partial \nu}\, \mathrm{d}\sigma+\int_{\Gamma_1}\dfrac{\partial\Tilde{P}}{\partial \nu} \, \mathrm{d}\sigma+\int_{\Gamma_2}\dfrac{\partial\Tilde{P}}{\partial \nu} \, \mathrm{d}\sigma<0,
\end{equation*}
this gives a contradiction with fact that $\Tilde{P}$ is harmonic on $\Omega$. 

The previous argument shows that at least one of the functions \( P \) or \( \tilde{P} \) is constant in \( \Omega \). If \( P \) is constant in \( \Omega \), the proof is complete. In the final step, we demonstrate that if \( \tilde{P} \) is constant in \( \Omega \), then \( P \) must also be constant. To this end, suppose \( \tilde{P} \) is constant in \( \Omega \). We already know that on $\Gamma_0$ it holds
\begin{equation*}
    \dfrac{\partial \Tilde{P}}{\partial \nu}= -(u_{\nu\nu}+1)\sin r\langle \nabla r, \nu \rangle.
\end{equation*}
Since $\dfrac{\partial \Tilde{P}}{\partial \nu}=0$ and $\sin r\langle \nabla r, \nu \rangle >0$ we conclude that $(u_{\nu\nu}+1)=0$, which together with \eqref{estrela} gives a contradition. 
\end{proof}

\subsubsection{Proof of Theorem \ref{cajarana}} 
We have proved in Proposition~\ref{z} that the auxiliary function  \( P \) is constant in $\Omega$. Consequently, \( \nabla^2 u \) is a multiple of the metric \( g \), specifically,  
\[
\nabla^2 u = \frac{\Delta u}{n} g = (-1 - u)g.
\]
Thus, using  the Obata-type lemma, we obtain \( \Gamma_0 = \partial B_R \cap \Sigma \), that is, \( \Omega = \{x\in \Sigma: R< d(x,\mathcal{O})< R_1\} \).  Moreover, \( u \) is a radial function, precisely, \eqref{radial} holds.
\

\subsubsection{General case}

 A vector field $X\in\mathfrak{X}(M)$ is called a conformal vector field if  
$$\langle\nabla_Y X, Z\rangle+\langle\nabla_Z X, Y\rangle=2\varphi\langle Y,Z\rangle,$$
where $\varphi$ is a smooth function on $M$.
Motivated by \cite{CV}, Proposition 3.1, we get the following result.
\begin{lemma}
Let $M$ be an $n$-dimensional Riemannian manifold with a conformal field $X$. If $u\in C^2(M)$, then
\begin{equation}\label{ca}
\operatorname{div}\left(\dfrac{|\nabla u|^2}{2}X-\langle X,\nabla u\rangle\nabla u \right)=\dfrac{n-2}{2}\varphi |\nabla u|^2 -\langle X, \nabla u \rangle\Delta u,
\end{equation}
where $\varphi$ is the conformal factor of $X$.
\end{lemma}

\begin{proof}
    Consider $f=\dfrac{|\nabla u|^2}{2}$ and $g=\langle X, \nabla u\rangle$. Note that
    \begin{eqnarray*}
        \operatorname{div}(fX-g\nabla v)&=& f\operatorname{div} X+\langle\nabla f, X\rangle-g\Delta u-\langle\nabla g, \nabla u\rangle\\
        &=& fn\varphi+\langle\nabla f, X\rangle-g\Delta u-\langle\nabla g, \nabla u\rangle
    \end{eqnarray*}
    and
    \begin{equation*}
        \langle\nabla f, X\rangle=\langle\nabla_X\nabla u, \nabla u\rangle= \nabla^2u(X,\nabla u).
    \end{equation*}
    Since $X$  is a conformal field we have 
    \begin{eqnarray*}
        \langle\nabla g,\nabla u\rangle&=&\nabla u\langle X, \nabla u\rangle\\
        &=&\langle\nabla_{\nabla u}X, \nabla u\rangle+\langle X, \nabla_{\nabla u}\nabla u\rangle\\
        &=& \varphi |\nabla u|^2+\langle\nabla_X\nabla u, \nabla u\rangle\\
        &=& 2f\varphi +\nabla^2u(X,\nabla v).
    \end{eqnarray*}
Thus,
    \begin{eqnarray*}
        \operatorname{div}(fX+g\nabla u)&=&fn\varphi + \nabla^2u(X,\nabla u)-g\Delta u-2f\varphi-\nabla^2u(X,\nabla u)\\
        &=&(n-2)f\varphi -g\Delta u,
    \end{eqnarray*}
that is,
\begin{equation*}
\operatorname{div}\left(\dfrac{|\nabla u|^2}{2}X-\langle X,\nabla u\rangle\nabla u \right)=\dfrac{n-2}{2}\varphi |\nabla u|^2 -\langle X, \nabla u \rangle\Delta u.
\end{equation*}
\end{proof}
  
Next result is a Pohozaev type identity addressed to annular sector-like domain in warped products.

\begin{proposition}\label{a11} 
  Consider $\Omega$ be an annular sector-like domain contained in a space form $M$, given by \eqref{08}, with $\Omega_2=B_R(\mathcal{O})$.
   If  there exists a classical solution $u \in C^2(\Omega)\cap C^1(\Gamma_0\cup\Gamma_1\cup\Gamma_2)$ of \eqref{PG} such that $u\in W^{1,\infty}\cap W^{2,2}(\Omega)$, then 
   \begin{multline}\label{re}
      {c_0^2}\int_{\Omega} \varphi \, \mathrm{d} V={(n+2)}\int_\Omega \varphi u \, \mathrm{d} V+2k\int_\Omega \varphi u^2\, \mathrm{d} V +\dfrac{n-2}{n}\int_\Omega u\langle \nabla u,\nabla \varphi \rangle \, \mathrm{d} V \\
      \left[2a+a^2k-\dfrac{c_0^2-c_2^2}{n}\right]|\Gamma_2|\psi(R)-\dfrac{n-2}{n}ac_2|\Gamma_2|\varphi(R),
   \end{multline}
   where $X=\psi\partial_t$ and $\operatorname{div} X=n\varphi.$
\end{proposition}
\begin{proof}  Since $X=\psi \partial_t$ is a closed conformal field, then $\operatorname{div} X=n\varphi$, where $\varphi=\psi'$.

      Since $\frac{\partial u}{\partial \nu}=0$ and $\langle X,\nu\rangle =0$ on $\Gamma_1$ and $\nu=-\partial_t$ on $\Gamma_2$ it follows from Theorem~\ref{DT} that
      \begin{eqnarray*}
          \int_\Omega \operatorname{div}\left(\dfrac{|\nabla u|^2}{2}X-\langle X,\nabla u\rangle\nabla u\, \right)\mathrm{d} V&=& \int_{\partial \Omega} \dfrac{|\nabla u|^2}{2}\langle X,\nu\rangle \, \mathrm{d} \sigma-\int_{\partial \Omega} \langle X,\nabla u\rangle \langle \nabla u, \nu\rangle \, \mathrm{d} \sigma\\
      &=&          -\dfrac{c_0^2}{2}\int_{\Gamma_0} \langle X,\nu\rangle \, \mathrm{d} \sigma- \dfrac{c_2^2}{2}\int_{\Gamma_2} \langle X,\nu\rangle \, \mathrm{d} \sigma\\
      &=&  -\dfrac{c_0^2}{2}\int_{\partial \Omega} \langle X,\nu\rangle \, \mathrm{d} \sigma+\dfrac{c_0^2-c_2^2}{2}\int_{\Gamma_2}\langle \psi \partial_t,\nu \rangle \,\mathrm{d}\sigma\\
      &=& -\dfrac{nc_0^2}{2}\int_{\Omega} \varphi \, \mathrm{d} V-\dfrac{c_0^2-c_2^2}{2}|\Gamma_2|\psi (R).
      \end{eqnarray*}
On the other hand, 
\begin{equation*}
    \int_\Omega \operatorname{div}(\varphi u \nabla u) \, \mathrm{d} V=\int_\Omega \varphi u \Delta u \, \mathrm{d} V+\int_\Omega u\langle \nabla u,\nabla \varphi \rangle \, \mathrm{d} V+\int_\Omega \varphi|\nabla u|^2 \, \mathrm{d} V.
\end{equation*}
Therefore, based on the overdetermined conditions \eqref{PG} we have that
      \begin{equation}\label{fa}
          \int_\Omega \varphi|\nabla u|^2 \, \mathrm{d} V = ac_2|\Gamma_2|\varphi(R)-\int_\Omega\left[\varphi u \Delta u+u\langle \nabla u,\nabla \varphi \rangle\right] \, \mathrm{d} V.
      \end{equation}
      By \eqref{ca} and \eqref{fa},
      \begin{eqnarray}\label{do}
          \dfrac{nc_0^2}{2}\int_{\Omega} \varphi \, \mathrm{d} V&=&\dfrac{n-2}{2}\int_\Omega\left[\varphi u \Delta u+u\langle \nabla u,\nabla \varphi \rangle\right] \, \mathrm{d} V-\dfrac{n-2}{2}ac_2|\Gamma_2|\varphi(R)\nonumber\\
          &&+\int_\Omega \langle X,\nabla u\rangle \Delta u \, \mathrm{d} V-\dfrac{c_0^2-c_2^2}{2}|\Gamma_2|\psi (R).
      \end{eqnarray}

      Note that
      \begin{equation*}
          \int_\Omega \varphi u\Delta u \, \mathrm{d} V=-n\int_\Omega \varphi u \, \mathrm{d} V-nk\int_\Omega \varphi u^2\, \mathrm{d} V
      \end{equation*}
      and 
      \begin{eqnarray}
          \int_\Omega \langle X,\nabla u\rangle \Delta u \, \mathrm{d} V &=& -n  \int_\Omega \langle X,\nabla u\rangle \, \mathrm{d} V-nk \int_\Omega u\langle X,\nabla u\rangle \, \mathrm{d} V\nonumber\\
          &=& -n\int_\Omega (\operatorname{div}(uX)-u\operatorname{div}(X)) \, \mathrm{d} V-nk \int_\Omega u\langle X,\nabla u\rangle \, \mathrm{d} V\nonumber\\
          &=&an|\Gamma_2|\psi(R)+n^2\int_\Omega u\varphi \, \mathrm{d} V- \dfrac{nk}{2}\int_\Omega (\operatorname{div}(u^2X)-n\varphi u^2)\, \mathrm{d} V\nonumber\\
          &=&an|\Gamma_2|\psi(R)+n^2\int_\Omega u\varphi \, \mathrm{d} V+\dfrac{a^2nk}{2}|\Gamma_2|\psi(R)+\dfrac{n^2k}{2}\int_\Omega \varphi u^2 \, \mathrm{d} V.
      \end{eqnarray}
Therefore
\begin{eqnarray*}
          {c_0^2}\int_{\Omega} \varphi \, \mathrm{d} V&=&{(n+2)}\int_\Omega \varphi u \, \mathrm{d} V+2k\int_\Omega \varphi u^2\, \mathrm{d} V +\dfrac{n-2}{n}\int_\Omega u\langle \nabla u,\nabla \varphi \rangle \, \mathrm{d} V\nonumber\\
          &&+2a|\Gamma_2|\psi(R)+{a^2k}|\Gamma_2|\psi(R)-\dfrac{c_0^2-c_2^2}{n}|\Gamma_2|\psi (R)-\dfrac{n-2}{n}ac_2|\Gamma_2|\varphi(R)
      \end{eqnarray*}
and \eqref{re} holds.

\begin{remark}
We point out that if we have a $\Gamma_2=\emptyset$, we obtain a Pohozaev-type identity for cones, see, for example, \cite[Lemma 2.2]{araujo2025serrin}.
\end{remark}


\subsubsection{Proof of Theorem~\ref{general case}}
 Using an argument analogous to Proposition \ref{max} we can prove that 
    \begin{equation}\label{vento}
        P< \dfrac{\sin^2 {R_1}}{\cos^2{R_1}} \qquad in \quad \Omega.
    \end{equation}
    Let us consider the following field 
    \begin{equation*}
        X=\psi\partial_r
    \end{equation*}
    where $\psi(r)=\sin{r}$. Note that, 
    \begin{equation}\label{ax}
        \psi'(r)=\cos{r}>0, \quad \psi''=-\psi \quad and \quad \operatorname{div} X=n\psi' \quad on \quad \mathbb S^n_+. 
    \end{equation}
    Multiplying \eqref{vento} by $\psi'$ and integrating, we have,
    \begin{equation}\label{modo}
        \dfrac{\sin^2 {R_1}}{\cos^2{R_1}}\int_\Omega \psi'\, \mathrm{d} V> \int _{\Omega} P\psi' \mathrm{d} V=\int_\Omega \left(|\nabla u|^2\psi'+u^2\psi'+2u\psi'\right)\, \mathrm{d}V.
    \end{equation}
    Observe that,
    \begin{equation*}
        \langle \nabla \psi',\nabla u\rangle=\langle \psi''\nabla r,\nabla u\rangle = \psi'' u_r \qquad in \quad \Omega
    \end{equation*}
    and
    \begin{equation*}
        \langle X, \nu\rangle =\psi\langle \partial_r,\nu\rangle =0 \qquad on \ \Gamma_1.
    \end{equation*}
    Now, applying Theorem~\ref{DT}, we have

\begin{eqnarray*}
    \int_\Omega |\nabla u|^2\psi'\, \mathrm{d}V &=& \int_\Omega \langle \nabla (u\psi'),\nabla u\rangle\, \mathrm{d}V-\int_\Omega u\langle \nabla \psi',\nabla u\rangle\, \mathrm{d}V\nonumber\\
    &=& \int_\Omega \operatorname{div}(u\psi'\nabla u)\, \mathrm{d}V- \int_\Omega u\psi'\Delta u\, \mathrm{d}V- \int_\Omega u\psi''u_r\, \mathrm{d}V\nonumber\\
    &=&  \int_{\partial\Omega}  u\psi'\dfrac{\partial u}{\partial\nu}\, \mathrm{d}\sigma- \int_\Omega u\psi'(-n-nu)\,\mathrm{d}V+ \int_\Omega u\psi u_r \, \mathrm{d}V\nonumber\\
    &=&  \int_{\Gamma_2} u\psi'\dfrac{\partial u}{\partial\nu} \, \mathrm{d}\sigma+n\int_\Omega u\psi' \, \mathrm{d}V+n\int_\Omega u^2\psi' \, \mathrm{d}V+\int_\Omega u\psi u_r \, \mathrm{d}V\nonumber
\end{eqnarray*}
which together with the boundary conditions gives 
\begin{equation}\label{a6}
      \int_\Omega |\nabla u|^2\psi'\, \mathrm{d}V =
 \dfrac{a\cos R\sin R}{\cos R_1}|\Gamma_2|+n\int_\Omega u\psi' \, \mathrm{d}V+n\int_\Omega u^2\psi' \, \mathrm{d}V+\int_\Omega u\psi u_r \, \mathrm{d}V
\end{equation}

    since $u=0$ on $\Gamma_0$ and $\partial u/\partial \nu =0$ on $\Gamma_1$. By \eqref{ax}, have that
    \begin{eqnarray*}
        \int_\Omega u \psi u_r \, \mathrm{d} V&=& \int_\Omega u\langle \nabla u, \psi\partial_r\rangle \, \mathrm{d} V=\dfrac{1}{2}\int_\Omega \langle \nabla u^2, X\rangle\,  \mathrm{d} V\\
&=& \dfrac{1}{2}\int_\Omega \operatorname{div}(u^2 X )\, \mathrm{d} V-\dfrac{1}{2}\int_\Omega u^2\operatorname{div}( X )\, \mathrm{d} V\\
&=& \dfrac{1}{2}\int_{\partial \Omega} u^2\langle X,\nu\rangle\,  \mathrm{d} \sigma -\dfrac{n}{2}\int_\Omega u^2 \psi'\,  \mathrm{d} V\\
&=& -\dfrac{a^2}{2}\sin{R}|\Gamma_2|-\dfrac{n}{2}\int_\Omega u^2 \psi'\,  \mathrm{d} V.
    \end{eqnarray*}
Thus,
    \begin{equation}\label{a7}
        \int_\Omega u^2 \psi ' \mathrm{d} V=-\dfrac{a^2}{n}\sin{R}|\Gamma_2|-\dfrac{2}{n}\int_\Omega u \psi u_r \mathrm{d} V
    \end{equation}
    Replacing \eqref{a6} and \eqref{a7} in \eqref{modo}, we get
    \begin{equation}\label{b2}
    \begin{aligned}
         \dfrac{\sin^2{R_1}}{\cos^2{R_1}}\int_\Omega \psi'\, \mathrm{d} V &> a\sin{R}|\Gamma_2|\dfrac{\cos{R}}{\cos{R_1}}-a^2\sin{R}|\Gamma_2|-2\int_\Omega u \psi u_r\, \mathrm{d} V +n\int_\Omega u\psi' \,  \mathrm{d} V\\
        &+\int_\Omega u\psi u_r\, \mathrm{d} V-\dfrac{a^2}{n}\sin{R}|\Gamma_2|- \dfrac{2}{n}\int_\Omega u\psi u_r \, \mathrm{d} V+ 2\int_\Omega u\psi'\, \mathrm{d} V\\
        &= a\sin{R}|\Gamma_2|\left(\dfrac{\cos{R}}{\cos R_1}-\dfrac{a(n+1)}{n} \right)+(n+2)\int_\Omega u\psi'  \, \mathrm{d} V\\
        &-\left(1+\dfrac{2}{n}\right) \int_\Omega u\psi u_r \, \mathrm{d} V.
    \end{aligned}
            \end{equation}
By Lemma~\ref{a11},
\begin{equation}\label{b1}
    \begin{aligned}
           \dfrac{\sin^2{R_1}}{\cos^2{R_1}}\int_{\Omega} \psi' \, \mathrm{d} V&={(n+2)}\int_\Omega \psi' u \, \mathrm{d} V-\left(1+\dfrac{2}{n}\right) \int_\Omega u\psi u_r \, \mathrm{d} V\\
          &+2a|\Gamma_2|\sin(R)+\left(1-\dfrac{2}{n}\right){a^2}|\Gamma_2|\sin(R)-{\dfrac{\sin^2{R_1}-\sin^2{R}}{n\cos^2{R_1}}}|\Gamma_2|\sin (R)\\
          &-\left(1-\dfrac{2}{n}\right)a\sin{R}|\Gamma_2|\dfrac{\cos{R}}{\cos{R_1}}. 
\end{aligned}
\end{equation}
Therefore, through \eqref{b2} and \eqref{b1}, we have
    \begin{equation*}
    \begin{aligned}
          a\left(\dfrac{\cos{R}}{\cos R_1}-\dfrac{a(n+1)}{n} \right)&<2a+\left(1-\dfrac{2}{n}\right){a^2}-{\dfrac{\sin^2{R_1}-\sin^2{R}}{n\cos^2{R_1}}}-\left(1-\dfrac{2}{n}\right)a\dfrac{\cos{R}}{\cos{R_1}}. 
\end{aligned}
\end{equation*}
Thus,
 \begin{equation*}
    \begin{aligned}
         0&<{\dfrac{\sin^2{R}-\sin^2{R_1}}{\cos^2{R_1}}}-2a(n-1)\dfrac{\cos{R}}{\cos{R_1}}+(2n-1)a^2+2an. 
\end{aligned}
\end{equation*}
But,
\begin{equation*}
    {\dfrac{\sin^2{R}-\sin^2{R_1}}{\cos^2{R_1}}}\leq -2a-a^2
\end{equation*}
and
\begin{equation*}
    a+1\leq  \dfrac{\cos{R}}{\cos{R_1}}.
\end{equation*}
Hence,
\begin{equation*}
    0< 2a(n-1)\left(a+1- \dfrac{\cos{R}}{\cos{R_1}}\right)\leq 0,
\end{equation*}
which is a contradiction. Thus, \( P \) is constant. Consequently, \( \nabla^2 u \) is a multiple of the metric \( g \), specifically,  
\[
\nabla^2 u = \frac{\Delta u}{n} g = (-1 - u)g.
\]
 Therefore, by Lemma~\ref{Obata}, the result follows. Moreover,
\begin{equation*}
    u(x) = \frac{\cos{r}-\cos{R_1}}{\cos{R_1}},
\end{equation*}
where \( r = \operatorname{dist}(x, p) \).

\subsection{Hyperbolic case}

\subsubsection{Proof of Theorem~\ref{jenipapo}} 

\begin{proposition}\label{egua} Let $u$ be a solution of Problem~\ref{saber3}. The $P-$function 
   \[
P(u) := |\nabla u|^2 + 2u - u^2
\] 
 satisfies the following properties:
\begin{equation}
\begin{cases}
    \Delta P \geq 0 & \text{in } \Omega, \\
    \dfrac{\partial P}{\partial \nu} \leq 0 & \text{on } \Gamma_1, \\
    P = \dfrac{\sinh^2 R_1}{\cosh^2 R_1} & \text{on } \Gamma_0, \\
    P = \dfrac{\sinh^2 R}{\cosh^2 R_1} + 2a - a^2 & \text{on } \Gamma_2.
\end{cases}
\end{equation}
Moreover, 
\begin{equation*}
    \Delta P =0 \Longleftrightarrow \nabla^2 u = -(1-u)g.
\end{equation*}
\end{proposition}
\begin{proof}
Indeed, it is known that
\[
(\Delta u)^2 \leq n |\nabla^2 u|^2,
\]
with equality holding if and only if \(\nabla^2 u\) is proportional to the metric \(g\). Applying the Bochner formula, we have:
\begin{align*} 
\Delta |\nabla u|^2 &= 2 \langle \nabla(\Delta u), \nabla u \rangle + 2 |\nabla^2 u|^2 + 2 \operatorname{Ric}(\nabla u, \nabla u) \\ 
&\geq 2n |\nabla u|^2 + \frac{2}{n}(\Delta u)^2 - 2(n-1)|\nabla u|^2  \\  
&= -2\Delta u + \Delta u^2.
\end{align*} 
We conclude that
\[
\Delta P \geq 0,
\]
indicating that \(P(u)\) is subharmonic in \(\Omega\). 

Again, from the convexity of the cone, we obtain  
\[
A(\nabla u, \nabla u) = \langle \nabla_{\nabla u} \nu, \nabla u \rangle \geq 0 \quad \text{on } \Gamma_1.
\]
Therefore, on \(\Gamma_1\),
\begin{align*}
    0 = \nabla u\left(\dfrac{\partial u}{\partial \nu}\right) &= \nabla u \langle \nabla u, \nu \rangle \\
    &= \langle \nabla_{\nabla u} \nabla u, \nu \rangle + \langle \nabla u, \nabla_{\nabla u} \nu \rangle \\ 
    &= \nabla^2 u(\nabla u, \nu) + A(\nabla u, \nabla u) \\
    &\geq \nabla^2 u(\nabla u, \nu).
\end{align*}
Thus,
\[
\dfrac{\partial P}{\partial \nu} = 2\nabla^2 u(\nabla u, \nu) + 2\dfrac{\partial u}{\partial \nu} - 2u\dfrac{\partial u}{\partial \nu} \leq 0 \quad \text{on } \Gamma_1,
\]
which complete the proof.\end{proof}

We claim that \(P\) is constant in \(\Omega\), and the constant \(a\) satisfies
\[
a = 1 - \frac{\cosh R}{\cosh R_1}.
\]
To prove this, we observe that, by Proposition~\ref{MP}, the maximum value of \(P\) is attained on \(\Gamma_0 \cup \Gamma_2\). 

From the hypothesis \(1 - \frac{\cosh R}{\cosh R_1} \leq a < 0\), we have
\[
(a - 1)^2 \leq \frac{\cosh^2 R}{\cosh^2 R_1}.
\]
Consequently,
\[
a^2 - 2a \leq \frac{\cosh^2 R}{\cosh^2 R_1} - 1 = \frac{\sinh^2 R - \sinh^2 R_1}{\cosh^2 R_1}.
\]
Therefore,
\[
\frac{\sinh^2 R}{\cosh^2 R_1} + 2a - a^2 \geq \frac{\sinh^2 R_1}{\cosh^2 R_1},
\]
which implies that \(P\) attains its maximum value on \(\Gamma_2\). Suppose, for contradiction, that \(P\) is not constant. By Hopf's lemma, we have
\begin{equation}\label{3.1}
    \frac{\partial P}{\partial \nu} > 0 \quad \text{on } \Gamma_2,
\end{equation}
where \(\nu\) is the outward unit normal to \(\Gamma_2\).

Let \(\{e_1, \ldots, e_{n-1}, \partial_r\}\) be a local orthonormal frame of \(\Omega\) at \(\Gamma_2\), where \(\{e_1, \ldots, e_{n-1}\}\) is tangent to \(\Gamma_2\). Since \(u\) is constant on \(\Gamma_2\), we have
\[
u_i = 0 \quad \text{and} \quad u_{ij} = 0 \quad \text{on } \Gamma_2,
\]
for all \(i, j = 1, \ldots, n-1\). Moreover, since \(u_r\) is constant on \(\Gamma_2\),
\[
u_{ri} = 0 \quad \text{on } \Gamma_2,
\]
for all \(i = 1, \ldots, n-1\). Then, for all \(W, Z \in \mathfrak{X}(\Omega)\),
\begin{equation}\label{2d}
    \nabla^2 u(W, Z) = \langle W, \partial_r \rangle \langle Z, \partial_r \rangle u_{rr} \quad \text{on } \Gamma_2.
\end{equation}
Analogously, 
\begin{equation}\label{2d*}
    \nabla^2 u(W,Z)=\langle W, \nu \rangle \langle Z , \nu\rangle u_{\nu \nu} \quad on \ \Gamma_0 ,
\end{equation}
where $\nu$ is orthogonal to $\Gamma_0$.

Let $\Delta^{\Gamma_2}$ be the Laplacian of $\Gamma_2$ and $H$ the mean curvature of $\Gamma_2$.
We recall that
\begin{equation*}
    \Delta u=\Delta^{\Gamma_2} u+H\langle\nabla u,\nu\rangle+\nabla^2u(\nu,\nu),
\end{equation*}
along  $\Gamma_2$.
Since $u$ is constant on $\Gamma_2$, we have $\Delta^{\Gamma_2}u=0$. Since  $\Gamma_2$ is a ball of radius $R$,  the mean curvature of $\Gamma_2$ is equal to $(n-1){\cosh R}/{\sinh R}$ and, therefore, 
 \begin{eqnarray*}
     \Delta u&=& \Delta^{\Gamma_2} u+H\langle \nabla u, \partial_r \rangle + \nabla^2 u(\partial_r, \partial_r)\\
            &=& (n-1)\dfrac{\cosh R}{\sinh R}u_r +u_{rr}, \quad on \ \Gamma_2.
 \end{eqnarray*}
 Substituting the known values, we obtain
\[
-n + na = (n - 1)\frac{\cosh R}{\sinh R}\left(-\frac{\sinh R}{\cosh R_1}\right) + u_{rr},
\]
which simplifies to
\[
u_{rr} = -n + na + (n - 1)\frac{\cosh R}{\cosh R_1} \quad \text{on } \Gamma_2.
\]

By \eqref{2d}, the normal derivative of \(P\) on \(\Gamma_2\) is given by
\begin{align*}
    \dfrac{\partial P}{\partial \nu} &= 2\nabla^2 u(\nabla u, \nu) + 2\dfrac{\partial u}{\partial \nu} - 2u\dfrac{\partial u}{\partial \nu}\\
    &=2u_{rr}\dfrac{\partial u}{\partial \nu}+2\dfrac{\partial u}{\partial \nu}-2u\dfrac{\partial u}{\partial \nu}\nonumber.
\end{align*}

Substituting the expression for \(u_{rr}\), we obtain
\begin{equation}\label{male}
    \dfrac{\partial P}{\partial \nu} = -2(n - 1)\frac{\sinh R}{\cosh R_1}\left(\frac{\cosh R}{\cosh R_1} - 1 + a\right) \quad \text{on } \Gamma_2.
\end{equation}
Under the hypothesis \(a - 1 + \frac{\cosh R}{\cosh R_1} \geq 0\), this implies
\[
\dfrac{\partial P}{\partial \nu} \leq 0 \quad \text{on } \Gamma_2,
\]
which contradicts \eqref{3.1}. Therefore, \(P\) must be constant, and
\[
a = 1 - \frac{\cosh R}{\cosh R_1}.
\]
As a result, \(\nabla^2 u\) is a multiple of the metric \(g\), specifically,
\[
\nabla^2 u = \frac{\Delta u}{n} g = (-1 + u)g.
\]
By Lemma~\ref{Obata}, the result follows. Moreover, the solution \(u\) is given explicitly by
\[
u(x) = \frac{\cosh R_1 - \cosh r}{\cosh R_1},
\]
where \(r = \operatorname{dist}(x, p)\).

\end{proof}

\subsubsection{Proof of Theorem~\ref{jenipapo2}}

    Consider the function
\[
P(u) := |\nabla u|^2 + 2u - u^2.
\]
As shown in the proof of Theorem~\ref{jenipapo}, the maximum value of \(P\) is attained on \(\Gamma_0 \cup \Gamma_2\). Moreover, \(P\) attains its maximum on \(\Gamma_0\), and
\begin{equation*}
    a - 1 + \dfrac{\cosh{R}}{\cosh{R_1}} \leq 0.
\end{equation*}
Thus, by \eqref{male}, we have
\begin{equation*}
    \dfrac{\partial P}{\partial \nu} \geq 0 \quad \text{on} \ \Gamma_2.
\end{equation*}
By applying the Hopf lemma to \(\Gamma_0\) and using \eqref{2d*}, we derive
\begin{equation*}
    0 < \dfrac{\partial P}{\partial \nu} = 2\nabla^2 u(\nabla u, \nu) + 2\dfrac{\partial u}{\partial \nu} + 2u\dfrac{\partial u}{\partial \nu} = 2\dfrac{\partial u}{\partial \nu} (u_{\nu\nu} + 1).
\end{equation*}
Thus, \( (u_{\nu\nu} + 1) > 0 \) on \(\Gamma_0\), since \(\dfrac{\partial u}{\partial \nu} > 0\) on \(\Gamma_0\).

Analogously, on \(\Gamma_2\), we have
\begin{eqnarray*}
    0 \leq \dfrac{\partial P}{\partial \nu} = 2\nabla^2 u(\nabla u, \nu) + 2\dfrac{\partial u}{\partial \nu} + 2u\dfrac{\partial u}{\partial \nu} = 2\dfrac{\partial u}{\partial \nu} (u_{rr} + 1 - a).
\end{eqnarray*}
Thus, \( (u_{rr} + 1 + a) \leq 0 \) on \(\Gamma_2\), since \(\dfrac{\partial u}{\partial \nu} < 0\) on \(\Gamma_2\).

\begin{claim}
    The function \(\widetilde{P} \coloneqq \langle \nabla u, \nabla \psi' \rangle - u \psi' + \psi'\) is harmonic in \(\Omega\), 
    where \(\psi = \sinh(r)\).
\end{claim}

To prove this claim, we begin by noting that
\[
\nabla^2 \psi' = \psi' g, \quad \Delta \psi' = n \psi',
\]
where \(g\) denotes the metric of \(S^n\). Using the polarized Bochner formula, we derive
\begin{align*}
\Delta \langle \nabla u, \nabla \psi' \rangle &= \langle \nabla (\Delta u), \nabla \psi' \rangle + \langle \nabla u, \nabla (\Delta \psi') \rangle + 2\operatorname{tr}(\nabla^2 u \circ \nabla^2 \psi') + 2\operatorname{Ric}(\nabla u, \nabla \psi') \\
&= 2\psi'(n u - n) + 2 \langle \nabla u, \nabla \psi' \rangle \\
&= 2n u \psi' - 2n \psi' + 2 \langle \nabla u, \nabla \psi' \rangle.
\end{align*}
Since
\[
\Delta (u \psi') = u \Delta \psi' + \psi' \Delta u + 2\langle \nabla u, \nabla \psi' \rangle = 2n u \psi' - n \psi' + 2\langle \nabla u, \nabla \psi' \rangle,
\]
it follows that
\[
\Delta (\langle \nabla u, \nabla \psi' \rangle - u \psi') = -n \psi' = -\Delta \psi'.
\]
This establishes that \(\widetilde{P}\) is harmonic in \(\Omega\).

For the sake of contradiction, assume that the \(\widetilde{P}\)-function is not constant. The next step is to analyze the sign of the normal derivative of the \(\widetilde{P}\)-function on the boundary of \(\Omega\). From this, we arrive at a contradiction:
\[
\dfrac{\partial \widetilde{P}}{\partial \nu} = \nabla^2 u(\nabla \psi', \nu) + \nabla^2 \psi'(\nabla u, \nu) + u \dfrac{\partial \psi'}{\partial \nu} + \psi' \dfrac{\partial u}{\partial \nu} + \dfrac{\partial \psi'}{\partial \nu}.
\]
Thus, since \(\Omega\) is a star-shaped domain with respect to \(\mathcal{O}\), on \(\Gamma_0\) we have
\begin{eqnarray*}
    \dfrac{\partial \widetilde{P}}{\partial \nu} &= u_{\nu \nu} \dfrac{\partial \psi'}{\partial \nu} - \psi' \dfrac{\partial u}{\partial \nu} + u \dfrac{\partial \psi'}{\partial \nu} + \psi' \dfrac{\partial u}{\partial \nu} + \dfrac{\partial \psi'}{\partial \nu} \\
    &= (u_{\nu\nu} + 1) \dfrac{\partial \psi'}{\partial \nu} \\
    &= (u_{\nu\nu} + 1) \sinh r \langle \nabla r, \nu \rangle < 0.
\end{eqnarray*}

Analogously, on \(\Gamma_2\) we obtain
\begin{eqnarray*}
    \dfrac{\partial \widetilde{P}}{\partial \nu} &= u_{\nu \nu} \dfrac{\partial \psi'}{\partial \nu} - \psi' \dfrac{\partial u}{\partial \nu} + u \dfrac{\partial \psi'}{\partial \nu} + \psi' \dfrac{\partial u}{\partial \nu} + \dfrac{\partial \psi'}{\partial \nu} \\
    &= (u_{rr} + 1 + a) \dfrac{\partial \psi'}{\partial \nu} \\
    &= (u_{rr} + 1 + a) \sinh r \langle \nabla r, \nu \rangle \leq 0 \quad \text{on} \ \Gamma_2,
\end{eqnarray*}
and on \(\Gamma_1\), because \(\nabla_{\nabla \psi'} \nu = 0\) and \(\langle \nabla \psi', \nu \rangle = 0\), we conclude that
\begin{eqnarray*}
    \dfrac{\partial \widetilde{P}}{\partial \nu} &= \nabla \psi' \left( \dfrac{\partial u}{\partial \nu} \right) - \psi' \dfrac{\partial u}{\partial \nu} + u \dfrac{\partial \psi'}{\partial \nu} + \psi' \dfrac{\partial u}{\partial \nu} + \dfrac{\partial \psi'}{\partial \nu} = 0.
\end{eqnarray*}
Following the same argument as in the proof of Proposition~\ref{z} and Theorem~\ref{cajarana}, we conclude the proof of Theorem~\ref{jenipapo2}.

\bigskip

\subsection*{Statements and Declarations}

\begin{flushleft}
 {\bf Funding:}  
 J. M. do \'O acknowledges partial support from CNPq through grants 312340/2021-4, 409764/2023-0, 443594/2023-6, CAPES MATH AMSUD grant 88887.878894/2023-00
and Para\'iba State Research Foundation (FAPESQ), grant no 3034/2021, 
J. de Lima acknowledges partial support  from  the Coordena\c c\~ao de Aperfei\c coamento de Pessoal de N\'ivel Superior - Brasil (CAPES) - Finance Code 001 and 
M. Santos acknowledges partial support from CNPq through grant 306524/2022-8.\\
 {\bf Ethical Approval:}  Not applicable.\\
 {\bf Competing interests:}  The authors declare that they have no competing interests or other conflicts that could be perceived as influencing the results and/or discussion presented in this paper. \\
 {\bf Author Contributions Statement:}    All authors contributed to the study conception and design. All authors performed material preparation, data collection, and analysis. The authors read and approved the final manuscript.\\
{\bf Availability of data and material:}  Not applicable.\\
{\bf Ethical Approval:}  All data generated or analyzed during this study are included in this article.\\
{\bf Consent to participate:}  All authors consent to participate in this work.\\
{\bf Conflict of interest:} The authors declare no conflict of interest. \\
{\bf Consent for publication:}  All authors consent for publication. \\
\end{flushleft}

\bigskip


\end{document}